\documentclass[10pt]{article}
\usepackage[total={6.7in, 9in}]{geometry}
\usepackage{graphicx}
\usepackage{amssymb}
\usepackage{amsthm}
\usepackage[dvipsnames]{xcolor}
\usepackage{mathtools}
\usepackage{stmaryrd}
\usepackage[colorlinks=true, linkcolor=blue, citecolor=red, urlcolor=magenta]{hyperref}
\usepackage[normalem]{ulem}
\usepackage{comment}
\usepackage{authblk}
\usepackage{cleveref}
\usepackage{mathrsfs}
\usepackage{subcaption}
\usepackage{bm} 
\usepackage[numbers]{natbib}

\newcommand\numberthis{\addtocounter{equation}{1}\tag{\theequation}}
\usepackage{titlesec}
\titleformat{\section}
  {\large\bfseries} 
  {\thesection}{1em}{}
\titleformat{\subsection}
  {\normalsize\bfseries} 
  {\thesubsection}{1em}{}

\newcommand{\ind}{\mathbf{1}}

\renewcommand{\P}{\mathbb{P}}
\newcommand{\Ee}{\mathbb{E}}
\newcommand{\R}{\mathbb{R}}

\newcommand{\Rdp}{\R_{\!\scriptscriptstyle\geq 0}^E}
\newcommand{\N}{\mathbb{N}}
\newcommand\G{\mathbf{G}}
\newcommand\W{\mathbf{W}}
\newcommand\vep{\varepsilon}

\newcommand{\Bs}{\mathcal{B}^*}
\newcommand{\B}{\mathcal{B}}
\renewcommand{\Re}{\mathcal{R}}
\newcommand{\Rv}{\mathbb{R}^{E}}
\newcommand{\M}{{\R^{E\times K}}}
\newcommand{\Mp}{{(\R_{\!\scriptscriptstyle\geq 0})^{E\times K}}}

\newcommand{\p}[1]{{\mathbb P}\left(#1\right)}
\renewcommand{\S}{\mathcal{S}}
\newcommand{\tS}{\tilde{\mathcal{S}}}
\renewcommand{\l}{\langle}
\renewcommand{\r}{\rangle}
\newcommand{\C}{\mathscr{C}}
\newcommand{\lf}{ \langle\!\!\langle}
\newcommand{\rf}{\rangle\!\!\rangle}

\theoremstyle{definition}
\newtheorem{hyp}{Assumption}
\newtheorem{rem}{Remark}
\newtheorem{nota}{Notation}
\newtheorem{thm}{Theorem}
\newtheorem{definition}[thm]{Definition}
\newtheorem{cor}[thm]{Corollary}
\newtheorem{lem}[thm]{Lemma}
\newtheorem{prop}[thm]{Proposition}

\newtheorem{ex}[thm]{Example}

\title{Stochastic neutral fractions and the effective population size}
\author[1]{Raphaël Forien}
\affil[1]{BioSP, INRAE, 84000 Avignon, France}
\author[2]{Emmanuel Schertzer}
\author[2]{Zsófia Talyigás}
\affil[2]{Faculty of Mathematics, University of Vienna, Oskar-Morgenstern-Platz 1, 1090 Wien, Austria}
\author[3]{Julie Tourniaire}
\affil[3]{Université Marie et Louis Pasteur, CNRS, LmB (UMR 6623), F-25000 Besançon, France.}

\begin{document}
\maketitle

\begin{abstract}
	The dynamics of a general structured population is modelled using a general stochastic differential equation (SDE) with an infinite decomposability property. This property allows the population to be divided into an arbitrary number of allelic components, also known as stochastic neutral fractions. When demographic noise is small, a fast-slow principle provides a general formula for the effective population size in structured populations. To illustrate this approach, we revisit several examples from the literature, including  expansion fronts.
\end{abstract}

\section{Introduction and main result}

\subsection{Effective population size for structured population models}
\label{sec:effective_popsize}

The Wright-Fisher model is an idealized framework in population genetics that models the variation
of allele frequencies due to the effect of random sampling alone (neutral evolution).
It is a discrete-time model with a fixed population size $N$. At time $t=0$, each
individual is assigned an allelic type indexed by $[K] := \lbrace 1, \ldots, K \rbrace$
and at every further generation, every individual picks a parent uniformly at
random from the previous generation, independently of each other, and inherits
its allelic type.

For $n\in\mathbb{N}_0$,
let $X^{N,i}_{n}\in[0,1]$ be the frequency of allelic type $i$ in the population at generation $ n $.
Assuming that the initial frequencies $(X^{N,k}_{0}; k\in[K])$ converge in law
to a non-degenerate limit $ x_0 \in [0,1]^K $
as the population size $N$ tends to infinity,
$$
	(X^{N,k}_{[tN]} ; k\in[K], t\geq0) \Longrightarrow (Z^{k}_{t} ; k\in[K], t\geq0)
$$
in distribution where the limit process is a $(K-1)$-dimensional standard Wright-Fisher (WF) diffusion,
that is,
a diffusion taking values in the simplex $\{(z_1,...,z_K)\in [0,\infty)^K:\sum z_k=1\}$ and characterised by its
infinitesimal generator
\begin{equation*}
	\mathcal{L} f(Z) = \frac{1}{2} \sum_{1 \leq i , j \leq K} \left( \delta_{ij} Z_i - Z_i Z_j \right) \frac{\partial^2 f}{\partial z_i \partial z_j}(Z).
\end{equation*}
The convergence holds in distribution for the usual Skorohod topology \cite{etheridge_mathematical_2011}.

An alternative way to understand the Wright–Fisher model is the
so-called genealogical approach.
At a given time horizon, one can sample $k$ individuals and trace backward in time their ancestral lineages.
This generates a random ultrametric tree with $k$ leaves. After rescaling time by $N$ (as before), this
random genealogical structure converges to the celebrated Kingman coalescent \cite{kingman1982coalescent}.

The Wright-Fisher model is a purely genetic model and a fundamental question in population
genetics is to understand how the genetic composition of a population may be
influenced by ecological constraints.
The effective population $N_e$ is a key concept in population genetics, representing
the size of a Wright-Fisher model that would experience the same level of genetic drift as the actual population under study.
The study of the effective population size for structured population models has
experienced a rapid growth, see e.g. \cite{nordborg1997structured,nordborg2002separation,laporte2002effective}
and the review of Charlesworth \cite{charlesworth2009effective} with an extensive list of publications related to this question.

The common approach to compute $N_e$ is to encode the genealogical structure of a population
as a structured coalescent model \cite{wakeley2009coalescent}. Assuming that migration
between classes is much faster than coalescence, one can use a separation of time scales to
show that the genealogical structure converges to the standard Kingman coalescent up to a time rescaling which encodes the effective population size.

While this backward-in-time approach is intuitive and well-documented, it can be problematic since a population model is usually defined by prescribing its dynamics forward in time. {In general,} it is difficult to infer the law of genealogies from these forward-in-time dynamics, especially when the number of individuals in the system is not fixed a priori.
The aim of the present article is to bypass the backward approach and to study the effective population size through the lens of stochastic  neutral fractions.

Neutral fractions were first introduced by Hallatschek and Nelson \cite{
	hallatschek2008gene} in the context of F-KPP fronts. In the same setting, Birzu et al. \cite{birzu_fluctuations_2018} used simulations of a stochastic expansion front with
Allee effects to characterize the coalescence time scale.
Related ideas have also been used in other contexts, for instance in \cite{billiard_stochastic_2015}
where the dynamics of neutral markers were studied under successive selective sweeps,
or in \cite{forien_stochastic_2022,forien_central_2025}, which analyzed the dynamics of neutral markers in a family of spatial population models.
These ideas can in fact be traced back to Malécot's approach of computing probabilities of identity in state as a way to quantify the expected shared ancestry among samples of individuals \cite{malecot_mathematiques_1948}.
In this article, we generalize the neutral fractions approach to a stochastic setting in order to capture the strength of genetic drift for general structured population models.
Tough~\cite{tough2023scaling} derived a closely related scaling result for an individual-based spatial model with fixed population size, namely the multicolor Fleming–Viot process.

\subsection{Neutral fractions and assumptions} We start directly from the ``ecological'' description of the population model. We assume that individuals are structured into distinct classes, which could represent ecologically relevant characteristics such as developmental stages or spatial locations.

Let $E=\{x_1,\dots,x_{|E|}\}$ denote the finite set of these classes.
Let $V^N_t$
be a weak solution to the $|E|$-dimensional stochastic differential equation
\begin{equation}
	\label{eq:SDE_tot_pop}
	dV^N_t= b(V^N_t)\ dt+\frac{1}{\sqrt{N}}a(V^N_t)\ dW_t,
\end{equation}
where $N$ is a large demographic parameter,
$W_t$ is a $|E|$-dimensional standard Brownian motion, and $b:\Rv\to\Rv$
and $a:\Rv\to \R^{E\times E}$ are smooth enough to guarantee existence and uniqueness
of a weak solution to the SDE (\ref{eq:SDE_tot_pop}) for all $N\in\N$. In this framework, $V_t\in\R^E$ represents the configuration of the population at time $t$ (i.e., $(V_t)_x$
measures the ``number'' of individuals of type $x$ at time $t$).

We assume that this SDE represents a {\it infinite population} model in the sense
that it can be decomposed into infinitely many neutral components.
Formally, our decomposability condition translates into a semilinear structure in the drift and the covariance.

\begin{hyp}[{$\infty$-decomposable SDE}]
	\label{hyp:decomposability}
	There exist $F:\Rv\to \R^{E\times E}$ and
	$C :\Rv\to \R^{(E\times E)\times E}$ such that,
	\begin{equation}
		\label{eq:decomposability}
		\forall v\in\Rv, \quad b(v)=F(v)v, \quad
		\text{and} \quad
		aa^*(v)=C(v)v
	\end{equation}
	where
	\begin{eqnarray*}
		F(v)v = \left(\sum_{x,y} F(v)_{xy} v_y\right)_{x\in E} \ \ \mbox{and} \ \ \
		C(v)w = \left(\sum_{z\in E} C_{xy,z} w_z\right)_{x,y\in E}.
	\end{eqnarray*}
	In addition, we assume that there exists $\sigma:\R^E\times \R^E\to \R^{E\times E}$
	such that
	\begin{equation} \label{def_C}
		\forall v,w\in \R^E, \quad \sigma(v,w)\sigma^*(v,w)  = C(v)w.
	\end{equation}
\end{hyp}

{\begin{ex}\label{ex:0}
	Take $a_{xy}(v)=\delta_{xy}\sqrt{g(v_x)v_x}$,
	with $g\geq 0$. We have $a a^*(v) = \mbox{Diag}(g(v)) v$, which can be rewritten as $a a^*(v)=C(v) v$, with $C_{xy,z}(v)= \delta_{xz} \delta_{yz} g(v_{z})$.
\end{ex}}

For more intuition on
$F(v)$ and $C(v)$, we refer to Section \ref{sect:asexual-populations}, where those
operators are explained in terms of the individual mean and a covariance matrix
in the context of individual-based models.  Informally, \( F(v)_{xy} \) represents
the rate at which an individual of type \( y \) produces offspring of type \( x \).
We will refer to $F(v)$ as the mean matrix at $v$. Similarly, \( C_{xy,z}(v) \) can
be interpreted as an instantaneous covariance matrix: it quantifies the covariance
between the numbers of offspring of types \( x \) and \( y \) produced by a single
individual of type \( z \). In particular, Example \ref{ex:0} can be thought of
as a population with local branching.
We refer to \eqref{eq:mean} and \eqref{eq:covariance} below for a more detailed discussion.

\begin{rem}
	Such a decomposition is not necessarily unique (i.e.~several choices of $ F $ and $ C $ are compatible with the same $ b $ and $ a $), and we will see that this choice affects the results.
	The questions of whether one decomposition is the \emph{right} one actually depends on the individual-based model that the SDE \eqref{eq:SDE_tot_pop} approximates.
		{ This will be further discussed in details in Section \ref{sect:unicity}}.

	For the purpose of this work, we assume that a choice of $ F $ and $ C $ satisfying Assumption~\ref{hyp:decomposability} has been made once and for all and study the resulting neutral fractions.
	As we shall see in the statement of the theorem, only the $F$-decomposition
	affects the convergence result, while the $(\sigma, C)$-decomposition does not.
\end{rem}

The intuitive idea of neutral fractions is to decompose the population into allelic types so that each individual will be characterized by an ``ecological character" valued in $E$ and an allelic type valued in $[K]$.   Formally, the $K$ neutral fractions $(U^{N,1},...,U^{N,K})$ are defined as
the solutions to the system of SDEs
\begin{equation}
	\label{eq:SDE_fraction}
	\forall k\in[K], \quad dU_t^{N,k}=F\left(\sum_{j=1}^K U_t^{N,j}\right)U_t^{N,k}dt
	+\frac{1}{\sqrt{N}} \ \sigma\left(\sum_{j=1}^{K} U_t^{N,j},U_t^{N,k}\right)dW_t^k,
\end{equation}
where the $W^k$'s are independent $|E|$-dimensional Brownian motions.

Assumption~\ref{hyp:decomposability} implies that $\sum_j U^{N,j}$ is a weak solution
to the SDE \eqref{eq:SDE_tot_pop}. Indeed, under Assumption~\ref{hyp:decomposability}
it is clear that the drift term is \textit{compatible} with the sum.
On the other hand, the covariance structure of the sum of the fractions reads
\begin{equation*}
	\sum_j \sigma\sigma^*\left(\sum_k U^{N,k},U^{N,j}\right)
	=\sum_j C\left(\sum_k U^{N,k}\right)U^{N,j}=aa^*\left(\sum_j U^{N,j}\right),
\end{equation*}
which is precisely the covariance matrix obtained for $V_t^N$ in \eqref{eq:SDE_tot_pop}.
In other words, Assumption~\ref{hyp:decomposability} guarantees that
$U^{N,1},\ldots,U^{N,K}$ can be interpreted as the population sizes of $K$ neutral
subfamilies of a population whose size evolves according to the SDE \eqref{eq:SDE_tot_pop}.

In this work, we are interested in the large population limit of the \textit{composition matrix}
\begin{equation}
	U^N_t:=(U_t^{N,1},\ldots,U_t^{N,K}),
\end{equation}
where the $(U^{N,k})_{k=1}^K$ are solutions to \eqref{eq:SDE_fraction},
when the solution to the deterministic dynamic
($N\to\infty$) associated to \eqref{eq:SDE_tot_pop} has an attractive
fixed point. More precisely, we will make the generic assumption of a stable ecological equilibrium.

\begin{hyp}[Regularity and Positivity]\label{A:regularity}
	The $(F,\sigma)$-decomposition from Assumption~\ref{hyp:decomposability}
	is such that
	\begin{enumerate}
		\item For every $K\in\N$ and $N\in\mathbb{R}^+$,
		      the system of SDEs \eqref{eq:SDE_fraction} admits a unique weak solution;
		\item The function $F$ is in {$\mathcal{C}^3(\Rv,\R^{E\times E})$}
		      and {$\sigma\in\mathcal{C}^1(\Rv\times \Rv,\R^{E\times E})$}.
		\item For all $v \in \Rdp$, $F(v)$ is a Metzler matrix; that is, all off-diagonal entries of $F(v)$ are non-negative.
	\end{enumerate}
\end{hyp}

\begin{rem}
	\label{remm:metzler}
	Assumption~\ref{A:regularity}-3 is a natural condition in the context of SDEs derived as scaling limits of individual-based models. Further details can be found in Section~\ref{sect:asexual-populations}.
\end{rem}

\begin{hyp}[Existence of a fixed point]
	\label{A1}
	There exists a vector $\tilde h\in\R_+^{E}$ such that
	\begin{equation*}
		b(\tilde h)=F(\tilde h)\tilde h=0.
	\end{equation*}
\end{hyp}
\begin{hyp}[Local stability]
	\label{A2}
	There exists an open neighbourhood $\B$ of $\tilde h$ such that
		{$\B\subset \R_+^E$} and,
	for all $v_0\in\B$,
	the solution to the deterministic system
	\begin{equation}
		\label{eq: deterministic}
		\dot{v}_t=b(v_t)=F(v_t)v_t,
	\end{equation}
	with initial condition $v_0$, converges to $\tilde h$ as $t\to\infty$.
	We further assume that the eigenvalues $(\tilde \lambda_i)$ of the Jacobian
	matrix $J\equiv J_b(\tilde h)$ of
	$b$ at $\tilde{h}$ are such that
	\begin{equation*}
		0>\Re(\tilde \lambda_1)\geq\Re(\tilde \lambda_2)\geq\ldots\geq\Re(\tilde \lambda_{|E|}).
	\end{equation*}
\end{hyp}
\begin{hyp}\label{A3}
	We assume that there exists $t_0>0$ such that
	\begin{equation*}
		\forall x,y\in E,
		\quad (e^{t_0F(\tilde h)})_{xy}>0.
	\end{equation*}
	This combined with Assumption~\ref{A:regularity}-3 implies that $0$ is
	the principal eigenvalue of $F(\tilde h)$ and that the corresponding
	eigenspace is spanned by
	$\tilde h$. Moreover there exists an associated left-eigenvector
	$h\in \R_+^E$ that can be normalised so that
	$$\left<h,\tilde h\right>=1.$$

\end{hyp}

\begin{rem}[Perron-Frobenius type condition]
	The definition of $h$ and $\tilde h$ is analogous to what is obtained in  standard Perron--Frobenius decompositions for
	mean offspring matrices of multitype Galton-Watson processes (see e.g.~\cite{athreya2012branching}) with the difference that $\tilde h$ is not the solution of a linear equation. As a result, it cannot be normalized to have unit $L^{1}$-norm.
\end{rem}

\begin{rem}\label{rem:explicit_construction}
	In Assumption~\ref{hyp:decomposability}, we assume that the noise correlation
	matrix $a$ admits a $(\sigma,C)$ decomposition
	satisfying~\eqref{eq:decomposability}-\eqref{def_C}.
	Here, we provide a simple construction of $\sigma$ under the assumption that,
	for $v \in \mathbb{R}^{E}$, the condition $v_{y}=0$ implies $a_{xy}(v)=0$.
	This assumption is natural, since $a_{xy}$ quantifies the strength of the
	noise coming from the $y$-component on the $x$-component, and it typically arises in
	limiting SDEs derived as approximations of individual-based models.
	As we shall see, this assumption guarantees the existence of functions $C$ and $\sigma$
	satisfying \eqref{eq:decomposability} and \eqref{def_C}. Once again,
	this decomposition is not unique.

	Consider the functions $R:\Rdp\to \R^{E\times E}$ and $A:\Rdp\to
		\R^{E\times E}$ defined by
	\[ \forall v\in \R^E,\; \forall x,y\in E, \quad
		R_{xy}(v) = \delta_{xy}\sqrt{v_y}, \quad A_{xy}(v)=\begin{cases}
			a_{xy}(v)/\sqrt{v_y} & v_y\neq 0 \\
			0                    & v_y=0.
		\end{cases}
	\]
	This construction ensures that, for all $v\in \R^E$, $a(v)=A(v)R(v)$. Next, we define
	$\sigma$ and $C$ such that,
	\[
		\forall v,w\in \Rdp, \quad \sigma(v,w)=A(v)R(w),
	\]
	and,
	\[
		\forall v\in \Rdp,\; \forall x,y,z\in E, \quad C_{xy,z}(v)=A_{xz}(v)A_{yz}(v).
	\]
	It is straightforward to verify that the functions defined above satisfy
	\eqref{eq:decomposability} and \eqref{def_C}.
\end{rem}

{\begin{ex}\label{ex:1}
	With the above construction, we obtain $A_{xy}(v) =\delta_{xy}\sqrt{g(v_x)}$
	in Example~\ref{ex:0}.
\end{ex}}

\subsection{Main results}

We now provide some intuitive explanation of the forthcoming result. In \cite{katzenberger1990solutions}, Katzenberger considered a generalization of our problem involving an SDE (not necessarily $\infty$-decomposable) whose deterministic component drives the SDE to an invariant stable manifold. Under a small-noise assumption, it is shown that, after accelerating time by a suitable factor, the system converges to a Markovian SDE constrained on the invariant manifold. In practice, the coefficients of the limiting dynamics are not fully explicit. A first approach to further understanding the limiting SDE was proposed by Parsons and Rogers  in \cite{parsons2017dimension}, where they provide an interpretation of the coefficients in terms of geometrical properties of the deterministic flow. In this paper, we will see that the algebraic structure of $\infty$-divisible populations allows for an explicit characterization of the SDE in terms of the Wright-Fisher diffusion.

The assumptions of the previous section imply that the deterministic component of the neutral fractions dynamics \eqref{eq:SDE_fraction} drives the system to the invariant manifold
consisting of all convex combinations of $\tilde h$ (see below in~\eqref{eq:def_gamma}). {Then the general limiting Markovian SDE of Katzenberger will translate to an SDE for the constants of this convex combination; that is, to an SDE for the proportions of the subpopulations within the total population. Our result says that the coefficients of this SDE are those of the Wright-Fisher diffusion.}

Before stating our result, we will need some preliminary definitions.
For $ \theta \in \R^{K} $ and $ v \in \R^E $, we write $ \theta \otimes v $ for the element of $ \M $ whose entries are given by
\begin{equation*}
	(\theta \otimes v)_{xj} = v_x \theta^j, \quad x \in E, j \in [K].
\end{equation*}
With this notation, the invariant manifold is given by
\begin{equation}
	\label{eq:def_gamma}
	\Gamma^K:=\left\lbrace \theta \otimes \tilde h : \sum_{k} \theta^k=1, \theta^k\geq0\right\rbrace.
\end{equation}
We define furthermore the probability distribution
$$
	\Pi_x \ := \ \tilde h_x h_x, \quad x\in E,
$$
and
$$
	n_{e}(x,y) \ := \ \left( \frac{1}{{\tilde h_x \tilde h_y}}\sum_{z} C(\tilde h)_{xy, z} {\tilde h_z} \right)^{-1}, \quad x,y\in E.
$$
Finally, define
\begin{eqnarray}
	\Sigma^2& := & \langle h, { a (\tilde h)a^*(\tilde h)} h\rangle = \sum_{x,y} \frac{\Pi_x\Pi_y}{n_{e}(x,y)}. \label{eq:Sigma-ancestral}
\end{eqnarray}
Interpretations of these quantities can be found in Sections~\ref{sect:ancestral}
and~\ref{sect:asexual-populations}.

\begin{thm}
	\label{th:main_result}
	Let $U_0$ be an element of
	$\M$ with non-negative entries such that
	$V_0:=\sum_{j}U_0^j\in \B$.
	Assume that for all $N\in\N$, $U_0^N=U_0$. Then
	\begin{equation*}
		(U^N_{\frac{tN}{\Sigma^2}}; t\in(0,T]) \Longrightarrow
		\ \ (\theta_{t} \otimes \tilde h; t\in(0,T]),  \ \ \mbox{in $D((0,T], \R^{E \times K})$}
	\end{equation*}
	where $(\theta_{t}; t\geq 0)$ is distributed as a standard $(K-1)$-dimensional
	Wright-Fisher diffusion with initial condition $({\pi^1}/\tilde{h},...,{\pi^K}/\tilde{h})$
	(here, we mean division coordinate by coordinate),
	where ${\pi^j\equiv\pi^j}(U_0)$ is defined as the limit as
	$ t \to \infty $ of the solution $\pi_t$ to the Cauchy problem
	\begin{equation*}
		\begin{cases}
			\frac{d}{dt}{\pi}_t= F(v_t)\pi_t, & \quad \pi_0=U_0^j, \\
			\frac{d}{dt}{v}_t=F(v_t)v_t       & \quad v_0=V_0.
		\end{cases}
	\end{equation*}

\end{thm}

\begin{rem}
	We write that the convergence takes place in $ D((0,T], \R^{E \times K}) $, which is equivalent to saying that the sequence of processes converges in $D([\varepsilon,T], \R^{E \times K})$ for any $ \varepsilon \in (0,T) $.
\end{rem}

Theorem~\ref{th:main_result} shows that the fluctuations of neutral fractions converge to a standard Wright-Fisher diffusion after rescaling time by $N/\Sigma^2$. From a biological standpoint, this entails that our model behaves at the limit as a Wright-Fisher model with an effective population size $N_e$ given by
\begin{equation}\label{eq:Ne}
	N_e = \frac{N}{\Sigma^2}.
\end{equation}

\begin{rem}[Effective vs.~census population size]
Assume that $aa^*(\tilde h)$ is invertible.
		One can check (using Cauchy-Schwarz inequality) that
		\[1=\l h,\tilde h\r^2\leq \l\tilde h, aa^*(\tilde h)^{-1} \tilde h\r \l h, aa^*(\tilde h)h\r.\]
		This shows that
		\[ N_e \leq  \l\tilde h, aa^*(\tilde h)^{-1} \tilde h\r N.\]
		In Example~\ref{ex:0} with $g\equiv 1$, this inequality reads
		\[N_e\leq \l\tilde h,1\r N,
		\]
        where the right-hand side is interpreted as the carrying capacity of the population. This is consistent with the 
      biology literature, where it is often observed~\cite{Frankham_1995} that the effective population size of a structured population is smaller than its census size.

\end{rem}

This result can be understood as a slow-fast principle. The fast ecological dynamics ensure that the types within each fraction are distributed according to $\tilde h/|
	\tilde h|_1$ whereas the relative proportions of each fraction evolve on a slower ``evolutionary" time
scale according to a Wright-Fisher diffusion. This is illustrated in Figure~\ref{figure}.
This generalizes the \emph{collapse of structure} evidenced in Wright's island model \cite{etheridge_mathematical_2011}, in which a population subdivided in discrete demes linked by migration is shown to behave asymptotically as a panmictic population with an effective population size.

\begin{figure}[ht]
	\centering
	\begin{subfigure}[b]{\linewidth}
		\centering
		\includegraphics[width=\linewidth]{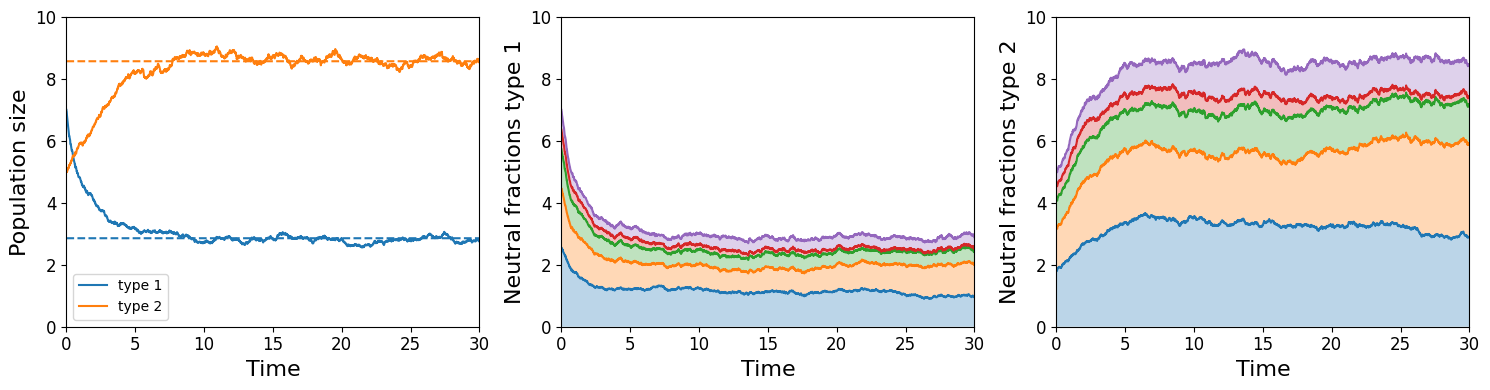}
		\caption{}
		\label{fig:N200_small_time}
	\end{subfigure}
	\begin{subfigure}[b]{\linewidth}
		\centering
		\includegraphics[width=\linewidth]{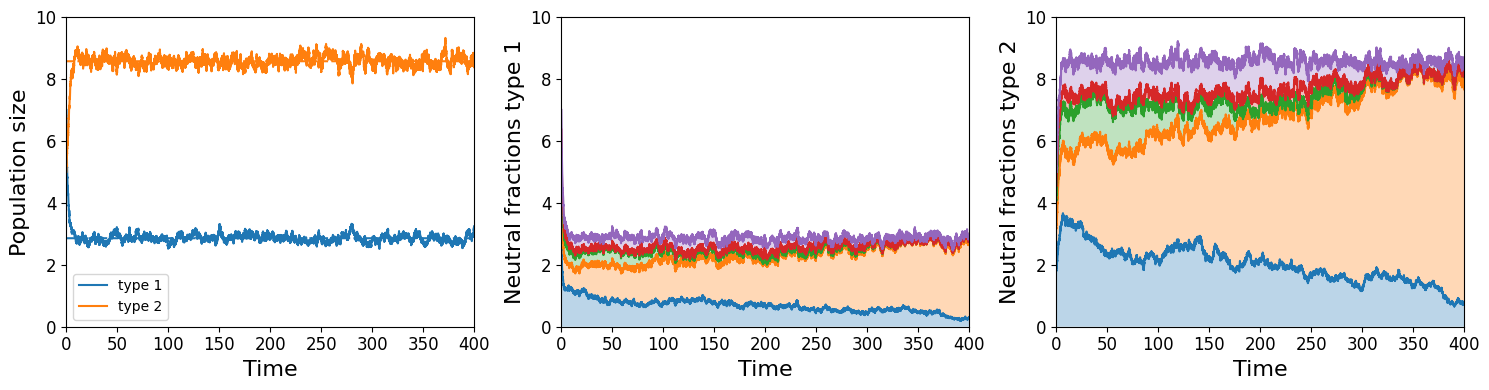}
		\caption{}
	\end{subfigure}
	\caption{One realization of the dynamics \eqref{eq:SDE_fraction} for $F$ as in \eqref{eq:ex_F} with
		$b_0 = 1$, $b_{11} = -0.2$, $b_{12} = -0.05$, $b_{22} = -0.1$, $b_{21} = -0.05$,
		$\alpha = \beta = 0.1$, $N = 200$, and $K = 5$ on (a) the ecological and (b) the evolutionary timescale.
		In the left panel, orange (resp.~blue) lines correspond to the population size of type 1 (resp.~type 2).
		Dashed lines show the corresponding deterministic equilibrium
		($\tilde{h} \approx (2.86, 8.57)$). The middle and right panels show the fraction dynamics
		for the first and second type, respectively.	}
	\label{figure}
\end{figure}

\subsection{Ancestral interpretation}
\label{sect:ancestral}
As already discussed in Section~\ref{sec:effective_popsize}, our approach allows to bypass the usual genealogical approach for understanding the effective population size. Interestingly, our forward approach bears a natural coalescent interpretation that allows us to conjecture the genealogical structure underpinning the forward dynamics \eqref{eq:SDE_tot_pop}.

Start from a structured coalescent where lineages take value in $E$ and migrate from $x$ to $y$ in $E$ at rate
$$
	m_{x,y} = \frac{ F(\tilde h)_{x,y}
		\tilde h_{y}}{\tilde h_{x}}.
$$
In addition, assume that a pair of lineages located respectively at $x$ and $y$ in $E$ coalesce independently at rate
$$
	c^N_{x,y} = \frac{1}{N n_{e}(x,y)}.
$$
A direct computation shows that the equilibrium distribution for a lineage is given by $\Pi$. Since coalescence rates are slow, lineages are asymptotically at equilibrium between two successive coalescence events. Thus, if one accelerates time by $N/\Sigma^2$, the structured coalescent converges to the (unstructured) Kingman coalescent. This can be rigorously proved along the same lines as Theorem 1 of \cite{nordborg1997structured}.

Note that the ancestral structure suggested by the neutral fractions is not entirely obvious from the forward dynamics. In particular, the underlying structured coalescent can allow for ``crossed coalescence" and extends beyond the structured coalescent models often used in the literature (see e.g., \cite{nordborg2002separation}) where coalescence only occurs when two lineages share the same type.

\subsection{Outline}

We first detail several implications and applications of our result in Section \ref{sec:IBM}.
In Section~\ref{sect:unicity}, we discuss the non-uniqueness of the decomposition
\eqref{eq:decomposability}.
In Section \ref{sect:neutral}, we state some general properties of stochastic neutral fractions.
Section~\ref{sec:proof} is dedicated to the proof of Theorem~\ref{th:main_result}.

\section{Examples}\label{sec:IBM}

In this section, we review several examples from the literature to illustrate how our result can be applied to various ecological scenarios.

\begin{itemize}
	\item In Section \ref{sect:asexual-populations}, we consider a general class of asexual populations. This will allow us to provide an intuitive interpretation of the various quantities introduced in Assumption~\ref{hyp:decomposability}.
	\item In Section \ref{sect:sexual-populations}, we show how our method can be extended to sexual populations.
	\item In the following two sections, we {\it formally} apply our methods to the case of an infinite state space (typically $E$ will be a subset of $\R$). While those results cannot be deduced from Theorem~\ref{th:main_result}, those two cases provide some natural conjectures for interesting biological scenarios and also call for a generalization of the stochastic neutral fractions approach in the context of infinite dimensional state space.  In these two cases, the expression of $\Sigma^2$ in (\ref{eq:Sigma-ancestral}) is replaced by an integral
	      of the type
	      \begin{equation}
		      \label{eq:variance_continuum}
		      \int_{E} \tilde r(x) \tilde h(x) h^2(x) dx.
	      \end{equation}
	      where $h,
		      \tilde h$ are principal eigenfunctions
	      of a linear differential operator and its adjoint and $\tilde r$ is a local reproduction rate.
	      One interesting aspect of infinite state spaces
	      is that
	      the above integral may not be finite, and a formal extension of Theorem \ref{th:main_result} may sometimes be impossible. In Section \ref{sect:F-KPP}, we present the interesting example of F-KPP fronts with Allee effect and present some recent results and conjectures when $\Sigma^2=\infty$.
\end{itemize}

\subsection{Asexual populations}
\label{sect:asexual-populations}
Population models in biology are often specified in terms  of stochastic individual based models.
We think about a general asexual population with frequency dependent reproduction rates.
The population's evolution is described by a family of rates
\begin{equation}\label{eq:branching-rates}
	(r(x, {\bf V}, {\bf n}); x\in E,  {\bf V}, {\bf n}\in \mathbb{N}_0^{E}),
\end{equation}
where $\mathbb{N}_0=\mathbb{N}\cup\{0\}$. Let ${\bf V}_t(x)$ be the number of individuals in class $x$ at time $t$. Given a population composition ${\bf V}_{t} = ({\bf V}_{t}(x); x\in E)$ at time $t$, we assume that every individual of type $x\in E$ dies and gives birth to ${\bf n}\in\mathbb{N}_0^{E}$ individuals at rate $r(x,{\bf V}_t, {\bf n})$.  Formally, this is a continuous time Markov chain in the space ${\mathbb N}_0^{E}$ with infinitesimal generator
\begin{equation}\label{eq:generator}
	G f({\bf V}) \ = \ \sum_{x\in E, {\bf n}\in \mathbb{N}_0^{E}}
	{\bf V}(x) r(x,{\bf V}, {\bf n}) \left(f({\bf V}+{\bf n} - e_{x}) - f({\bf V})\right),
\end{equation}
where $e_{x}$ is the vector of dimension $|E|$ with only $0$'s except for the $x$-th coordinate.

We assume that the branching process is parametrized by a demographic parameter $N$ so that
\begin{equation*}
	r^{(N)}(x, {\bf V}, {\bf n}) = \rho\left(x, \frac{{\bf V}}{N}, {\bf n}\right),
\end{equation*}
for some prescribed rates
\begin{equation*}
	(\rho(x, {\bf V}, {\bf n}); x\in E,  {\bf V}\in \mathbb{R}^{E}, {\bf n}\in \mathbb{N}_0^{E}).
\end{equation*}
For every ${\bf V}\in \R_+^E$, define the mean matrix
\begin{equation}
	\label{eq:mean}
	\ F({\bf V})_{xy}
	\ := \ \sum_{{\bf n}\in \mathbb{N}_0^{E}}	\rho\left(x,{\bf V},\bf{n}\right) ({\bf n}_y-\delta_{xy}), \quad x,y\in E,
\end{equation}
and
\begin{equation}\label{eq:covariance}
	\ C({\bf V})_{xy,z}:=\sum_{{\bf n}\in \mathbb{N}_0^{E}}	\rho\left(z,{\bf V},\bf{n}\right) \left({\bf n}_x-\delta_{zx}\right)
	\left({\bf n}_{y}-\delta_{zy}\right), \quad x,y,z\in E,
\end{equation}
assuming that the above sums are always well defined and finite.
Under this
definition, $F({\bf V})$ is Metzler; this justifies Remark~\ref{remm:metzler}.

Let $G^N$ be the generator of the rescaled process ${\bf V}/N$.
From the expression of the generator (\ref{eq:generator}) and a Taylor expansion in $N$, we get that, for any bounded $\mathcal{C}^2$ function $ f $,
\begin{eqnarray}
	\label{eq:generator-N}
	\forall\, {\bf V}\in \R^E, \ \ G^{N}f\left({\bf V}\right) \ = \
	\l F\left({\bf V}\right){\bf V}, \nabla f\left({\bf V}\right)  \r + \frac{1}{2 N} \sum_{x,y\in E}  \bigg(C\left({\bf V}\right){\bf V}\bigg)_{xy} \partial_{xy} f\left( {\bf V}\right) + o\left(\frac{1}{N}\right),
\end{eqnarray}
where
$$
	\forall x,y\in E, \quad
	\bigg (C\left({\bf V}\right){\bf V}\bigg)_{xy}
	\ = \ \sum_{z\in E} C\left({\bf V}\right)_{xy,z}  {\bf V}_z.
$$

If we ignore the $o(1/N)$ terms, the generator coincides with that of an infinite decomposable diffusion as in Assumption~\ref{hyp:decomposability}, suggesting
that the rescaled Markov process is well-approximated by the diffusion. A quantitative estimation for the approximation can be established along the same lines as Theorem 2.4.1.in~\cite{britton2019stochastic}.

\subsection{Sexual reproduction}
\label{sect:sexual-populations}

We now consider a continuous-time sexual population.
At time $t$, we denote by $M_t^N$ the number of males and $F^N_{t}$ the number of females. We assume that each pair of male-female mates at rate $\frac{1}{N}$; the resulting offspring is a male with probability $p\in(0,1)$ and a female with probability $1-p$. Finally, we assume that each individual dies at rate $\frac{\alpha}{N^2}(F^N_t+ M_t^N)^2$ where $\alpha$ and $N$ are both positive. More complex logistic models of sexual populations have been considered in \cite{castillo1995logistic}, but for the sake of simplicity, we will restrict to this specific model. However, the same method applies in the general case.

We are interested in the number of male and female genes at an autosomal locus of interest. For a diploid model, we  introduce the rescaled number of male and female genes
$$
	m_t^N := \frac{2 M_t^N}{N}, \ \ f_{t}^N := \frac{2 F_t^N}{N}.
$$
As in the previous example, we can compute the generator of the rescaled Markov process. A direct computation shows that, for every twice differentiable function $\rho : \R_+^2 \to \R$,
$$
	G^N \rho(m,f) \ =
	\ b_m(m,f) \partial_{m} \rho(m,f) + b_f(m,f) \partial_{f} \rho(m,f) +
	\frac{1}{2N}\left( a_m(m,f)  \partial_{mm} \rho +  a_m(m,f)  \partial_{ff} \rho
	\right) + o\left(\frac{1}{N}\right),
$$
where
$$
	b_m(m,f) = \frac{p}{2} m f -\frac{\alpha}{4} m(f+m)^2, \ \
	b_f(m,f) = \frac{1-p}{2} m f -\frac{\alpha}{4} f(f+m)^2,
$$
and
$$
	a_m(m,f) \ =\ pfm + \frac{\alpha}{2} m (f+m)^2, \ \ \ a_f(m,f) \ =\ (1-p)fm + \frac{\alpha}{2} f (f+m)^2.
$$
This suggests that the  discrete stochastic system
is well approximated by the SDE
$$d \left[\begin{array}{c}m_t \\ f_t\end{array}\right]
	\ = \ b(m_t,f_t)dt + \frac{1}{\sqrt{N}}
	\left(\begin{array}{cc} \sqrt{a_{m}(m_t,f_t)} & 0                     \\
             0                          & \sqrt{a_{f}(m_t,f_t)}
		\end{array}\right)dW_t
$$
where $W$ is a 2 dimensional standard Brownian motion. In order to apply our method, we impose a decomposition of $b(m,f)$ into the product $F(m,f) \left[\begin{array}{c}m \\ f\end{array}\right]$ that encapsulates the underlying biological mechanism. In the following, we choose
$$
	F(m,f) \ = \ \left( \begin{array}{cc} - \frac{\alpha}{4} (f+m)^2 + \frac{1}{2} \frac{ p f}{2} & \frac{p}{2} \frac{m}{2}                               \\
             \frac{1-p}{2} \frac{f}{2}                                    & -\frac{\alpha}{4} (f+m)^2 + \frac{1-p}{2} \frac{m}{2}
		\end{array}\right)
$$
reflecting the fact that when a new male or a new female is born,  genes are inherited equally from both parents.

We first note that the only stable equilibrium of the system is attained
at
$$
	m_{eq} \equiv \tilde h({m}) = p n_{eq}, \ \  f_{eq} \equiv \tilde h({f}) =  (1-p) n_{eq}  \ \ \mbox{where   $n_{eq} := 2\frac{p(1-p)}{\alpha}$.}
$$
Let $h$ be the left $0$-eigenvector of $F(\tilde h
	)$ such that $\langle \tilde h,h \rangle =1$. A direct computation shows that
$$
	h({m}) \ = \ \frac{1}{2 m_{eq}}, \ \ h({f}) \ = \ \frac{1}{2 f_{eq}},
$$
and it follows that
\begin{eqnarray*}
	\Sigma^2 & = & h(m)^2 a_m(m_{eq},f_{eq}) + h(f)^2 a_f(m_{eq},f_{eq})  \\
	& = & \left(\frac{1}{4 m_{eq}} + \frac{1}{4 f_{eq}}\right)\left(p(1-p) + \frac{1}{2}\alpha  n_{eq}\right) n_{eq}.
\end{eqnarray*}
Write
$$
	M_{eq} = \frac{N m_{eq}}{2},\quad  F_{eq} = \frac{N f_{eq}}{2},
$$
which is interpreted respectively as the number of males and females at equilibrium.
Then the latter relation can be rewritten as
$$
	\Sigma^2 \ =  N \left(
	\frac{1}{4 M_{eq}}+\frac{1}{4 F_{eq}}\right) s^2, \ \ \mbox{where $s^2= p(1-p)\frac{n_{eq}}{2} + \alpha (\frac{n_{eq}}{2})^2$},
$$
and where $s^2$ is interpreted as the infinitesimal variance for a single individual chosen uniformly at random in the population.
Finally, we obtain the following expression for the effective population size:
$$
	\frac{1}{N_e} \ = {s^2}\left( \frac{1}{4 M_{eq}}+\frac{1}{4 F_{eq}}\right).
$$
This can be compared to  \cite[Equation (1)]{charlesworth2009effective} for a discrete-time two-sexes Wright-Fisher model.

\subsection{Model of the gut microbiome}
\label{sec:gut}
We can also apply {formally} the previous approach to a model of the gut introduced in  \cite{labavic2022hydrodynamic} that includes hydrodynamic flow of food and bacteria and their interaction.

	{The gut is modeled as a one dimensional segment $[0,L]$. Nutrients are assumed to enter the gut segment at a constant rate, with no initial inflow of bacteria. At the exit, both nutrients and bacteria are freely transported out. The dynamics are governed by a constant flow velocity $v$ and diffusivity constant $D$. Bacteria harvest nutrients, following a Hill-type function
		characterized by a Monod constant $k$.}

As a result, the concentrations of food (or nutrients) $\mathrm{f}$ and bacteria $\mathrm{b}$
are described by a coupled system of stochastic partial differential equations:
\begin{eqnarray*}
	\partial_{t} \mathrm{f} & = & D \partial_{xx} \mathrm{f} - v \partial_{x} \mathrm{f} - \frac{r}{\alpha} \frac{\mathrm{b}\mathrm{f}}{k+\mathrm{f}},  \\
	\partial_t \mathrm{b} & = & D \partial_{xx} \mathrm{b} - v \partial_{x} \mathrm{b}+ r\frac{\mathrm{f}}{k+\mathrm{b}} \mathrm{b}+\sqrt{ \frac{1}{N} \gamma(\mathrm{b},\mathrm{f}) \mathrm{b} } \ \eta,
\end{eqnarray*}
with boundary conditions
\begin{align*}
	-D\partial_{x}[\mathrm{f},\mathrm{b}] + v[\mathrm{f},\mathrm{b}] \ = \ [v \mathrm{f}_{in},0],  \quad \text{at} \quad  x=0, \qquad
	D \partial_{x}[\mathrm{f},\mathrm{b}] \ = \ [0,0],   \quad \text{at }
	\quad x=L,
\end{align*}
and where {$N$ is a large demographic parameter related to the
		carrying capacity of the bacterial population, $\gamma(\mathrm{b},\mathrm{f})$ quantifies the strength
		of demographic fluctuations and
		$\eta$ is a space-time white noise.
		The absence of noise in the $\mathrm{f}$ component stems from the fact that, typically,
		the bacterial concentration is much smaller than the total colon content
		(see Model and Methods in~\cite{labavic2022hydrodynamic}). }

In \cite{labavic2022hydrodynamic}, the authors identified a critical speed

$$
	v^* = \sqrt{\frac{4rD}{\frac{k}{\mathrm{f}_{in}}+1}}
$$
above which bacteria are washed out so that for $v>v^*$ the only stationary solution for the bacterial population is $\tilde h\equiv 0$. In contrast, if $v<v^*$, there exists a non-zero stable stationary solution $\tilde h$. In this case, there also exists a stable non-trivial equilibrium $\tilde{\mathrm{f}}$ for the food concentration.
In the following, we will work under this condition.

To mirror the discrete approach, define the linear differential operator at equilibrium
$$
	{\cal L} u = D\partial_{xx} u -v \partial_x u + r\frac{\tilde{\mathrm{f}}}{k+\tilde h} u  \ \ \mbox{on $(0,L)$}
$$
with boundary conditions
\begin{align}
	D\partial_{x}u - v u \ = \ 0 ,\quad \text{at} \quad  x=0, \qquad
	\partial_{x}u \ = \ 0, \quad  \text{at} \quad        x=L.
	\label{eq:bc}
\end{align}
We can compute the adjoint of the operator by an integration by parts. For any
twice differentiable functions $f,g$ on $(0,L)$ and any $f$ satisfying the
boundary conditions (\ref{eq:bc}), we get

\begin{align}
	\l {\cal L}f,g \r \ = \ \l f,{\tilde{\cal L}}g \r + [Dg \partial_{x} f-D f\partial_{x} g - v f g ]_0^L \label{eq:bc2}
\end{align}
where
$$
	{\tilde{\cal L}} g = D\partial_{xx} g +v \partial_x g + r\frac{\tilde f}{k+\tilde h} g  \ \ \mbox{on $(0,L)$}.
$$
In (\ref{eq:bc2}), the boundary terms in the integration by parts yields
\begin{equation*}
	Df \partial_x g =0, \quad \text{at} \quad  x=0, \quad
	f(D\partial_x g +vg)=0, \quad \text{at} \quad  x=L,
\end{equation*}
where we used the fact that $f$ satisfies (\ref{eq:bc}).
In order to make the operator ${\tilde{\cal L}}$ the adjoint of $\mathcal L$, we need to remove the boundary conditions in the integration by part. Equivalently, this imposes the boundary conditions
\begin{align}
	\partial_x g=0,\quad \text{at} \quad  x=0, \qquad
	D\partial_xg+vg=0\quad    \text{at} \quad  x=L, \label{eq:eq:bc3}
\end{align}
on the differential operator ${\tilde{\cal L}}$. Let us now consider $h$ to be
the solution of the linear problem
$$
	\tilde{{\cal L}} h =  0 \ \ \mbox{on $(0,L)$ with boundary conditions (\ref{eq:eq:bc3})}.
$$
By linearity, solutions are always defined up to a multiplicative constant and we impose the Perron-Frobenius renormalization
$$
	\l \tilde h, h \r=1,
$$
in analogy with the discrete problem. From there, we can compute
\begin{equation*}
	\Sigma^2  =  \int_0^L \gamma(\tilde h, \tilde{\mathrm{f}})(x) \tilde h(x) h^2(x) dx, \qquad \text{ and } \qquad N_e = \frac{N}{\Sigma^2},
\end{equation*}
where we used the fact that ${\eta}$ is a space-time white noise.
	{Typically, the function $\gamma$ is taken to be the per-capita growth rate
		of the bacterial population, $i.e.,$ $\gamma(b,f)=\frac{rf}{k+b}$. With this choice of $\gamma$, the above formula for $\Sigma^2$ coincides with \eqref{eq:variance_continuum}.}

\subsection{F-KPP equation with Allee effect}\label{sect:F-KPP}

Consider the following partial differential equation:
\begin{equation} \label{kpp}
	\partial_t v = D \partial_{xx} v + f(v),
\end{equation}
where $ D > 0 $, $ f : \R_+ \to \R $ is continuously differentiable, $ f(0) = f(1) = 0 $ and $ v : \R_+ \times \R \to [0,1] $, and suppose that this equation admits a travelling wave solution taking the form
\begin{equation*}
	v(t,x) = \tilde{h}(x - ct),
\end{equation*}
where $ \tilde{h} : \R \to [0,1] $ connects the two equilibria 0 and 1, i.e. $ \lim_{x \to -\infty} \tilde{h}(x) = 1 $ and $ \lim_{x \to +\infty} \tilde{h}(x) = 0$, and $ c > 0 $.
Further assume that $ \tilde{h} $ is exponentially stable, meaning that there exists $ \lambda > 0 $ such that, for any $ v_0 $ sufficiently close to $ \tilde{h} $ (in a suitable sense), there exists $ \zeta \in \R $ and $ C > 0 $ such that
\begin{equation*}
	\| v(t, \cdot + ct) - \tilde{h}(\cdot - \zeta) \| \leq C e^{-\lambda t}.
\end{equation*}
This implies that the travelling wave solution $ \tilde{h} $ is a pushed front, see \cite{fife_approach_1977,rothe_convergence_1981}, and that $ c > c_{0} $, where
\begin{equation} \label{def:c0}
	c_{0} = 2 \sqrt{D f'(0)}
\end{equation}
is the speed of the travelling wave associated to the linear equation
\begin{equation*}
	\partial_t v = D \partial_{xx} v + f'(0) v.
\end{equation*}
(Note that if $ f'(0) < 0 $ then $ c_0 $ is not given by \eqref{def:c0} but is negative.)

We now consider the following stochastic perturbation of the PDE \eqref{kpp}, written in the reference frame moving at speed $ c $ to the right,
\begin{equation} \label{stochastic_kpp}
	\partial_t v^N_t = D \partial_{xx} v^N_t + c \partial_x v^N_t + f(v^N_t) + \sqrt{\frac{1}{N} v^N_t g(v^N_t)} \dot{W}(t),
\end{equation}
where $ g : \R_+ \to \R_+ $ is continuously differentiable and $ (W(t), t \geq 0) $ is a cylindrical Wiener process in $ \R $ (i.e. $ \dot{W} $ is space-time white noise).
We can then decompose $ v^N $ as a sum of subfamilies as follows.
Let $ r : \R_+ \to \R $ be such that $ f(v) = r(v) v $.
Then, given $ (u^{k,N}_0, k \in [K]) $ such that $ \sum_{j \in [K]} u^{j,N}_0 = v^N_0 $, let $ (u^{k,N}_t, k \in [K], t \geq 0) $ be a solution to the following system of stochastic partial differential equations
\begin{equation} \label{def:ukN_kpp}
	\partial_t u^{k,N}_t = D \partial_{xx} u^{k,N}_t + c \partial_x u^{k,N}_t + r\left( \sum_{j\in [K]} u^{j,N}_t \right) u^{k,N}_t + \sqrt{\frac{1}{N} u^{k,N}_t g\left( \sum_{j \in [K]} u^{j,N}_t \right)} \dot{W}^k(t), \quad k \in [K],
\end{equation}
where $ (W^k, k \in [K]) $ is an i.i.d.~family of cylindrical Wiener processes on $ \R $.
Then $ \sum_{k \in [K]} u^{k,N} $ is distributed as $ v^N $.

This system is somewhat more complicated than \eqref{eq:SDE_fraction}, first because it is an infinite-dimensional system, and because the associated deterministic dynamics (when $ N \to \infty $) admits a line of equilibria consisting of all possible shifts of the travelling wave profile $ \tilde{h} $.
Rigorously generalising the above results to such a setting is beyond the scope of the present paper, but let us pretend that it can be done and compare the predictions with what has been conjectured in the literature.

In an upcoming paper \cite{etheridge_fluctuations_2025}, it is already proven rigorously that, for a suitable initial condition and a bistable reaction term $ f $, $ v^N_t $ stays close to a shift of the travelling wave profile on time intervals of the form $ [0,NT] $, i.e. there exists a real-valued process $ (\xi^N_t, t \geq 0) $ such that
\begin{equation*}
	\sup_{t \in [0,NT]} \| v^N_{t} - \tilde{h}(\cdot - \xi^N_t) \| \to 0,
\end{equation*}
in probability as $ N \to \infty $.
Moreover, they show that $ (\xi^N_{Nt}, t \geq 0) $ converges in distribution to Brownian motion with constant (negative) drift as $ N \to \infty $.

Combining this with results on the long-time behaviour of neutral fractions in the corresponding deterministic system \cite{roques2012allee}, we can conjecture that the system $ (u^{k,N}_t, k \in [K], t \geq 0) $ stays close to a convex combination of a shift of the travelling wave profile, i.e. there exists $ (\theta^{k,N}(t), k \in [K], t \geq 0) $ taking values in $ \lbrace (\theta_k)_{k \in [K]} \in [0,1]^{K} : \sum_{k \in [K]} \theta_k = 1 \rbrace $ such that
\begin{equation}
	\label{eq:conj_fkpp}
	\sup_{t \in [0, NT]} \sum_{k \in [K]} \| u^{k,N}_t - \theta^{k,N}(t) \tilde{h}(\cdot - \xi^N_t) \| \to 0
\end{equation}
as $ N \to \infty $.

Let us now present what a generalisation of Theorem~\ref{th:main_result} would lead us to expect.
Let $ \mathcal{L} $ denote the second-order differential operator
\begin{equation*}
	\mathcal{L}v = D \partial_{xx} v + c \partial_x v + r(\tilde{h}) v.
\end{equation*}
Then $ \mathcal{L}\tilde{h} = 0 $ by the definition of $ \tilde{h} $.
Let $ h : \R \to \R_+ $ be such that $ \mathcal{L}^* h = 0 $ and $ \langle h, \tilde{h} \rangle = 1 $.
In this case, $ h $ can be written in terms of $ \tilde{h} $ as follows
\begin{equation} \label{expression_h}
	h(x) = \frac{\tilde{h}(x) e^{\frac{c x}{D}}}{\int_\R \tilde{h}(y)^2 e^{\frac{c y}{D}} dy}.
\end{equation}
The expected generalisation of Theorem~\ref{th:main_result} is the following.
As $ N \to \infty $, $ (\theta^{k,N}(Nt / \Sigma^2), t \geq 0) $ converges in distribution to a standard Wright-Fisher diffusion, where
\begin{equation*}
	\Sigma^2 = \int_\R h(x)^2 \tilde{h}(x) g(\tilde{h}(x)) dx.
\end{equation*}
{As in \Cref{sec:gut}, $g$ is typically taken as the per-capita growth rate $r$, so that we
recover \eqref{eq:variance_continuum}, where \(\tilde r = r \circ \tilde h\) is
interpreted as the local reproduction rate at equilibrium.}
Substituting \eqref{expression_h} in the above expression yields
\begin{equation*}
	\Sigma^2 = \frac{\int_\R \tilde{h}(x)^3 g(\tilde{h}(x)) e^{2 \frac{cx}{D}} dx}{\left( \int_\R \tilde{h}(x)^2 e^{\frac{cx}{D}} dx \right)^2}.
\end{equation*}
Note that this is exactly what was conjectured in \cite{birzu_fluctuations_2018} (Equation 5) for the rate of decrease of heterozygosity in fully pushed fronts.
Moreover, in \cite{schertzer2023spectral}, an analogue of the quantity $\Sigma^2$ was derived for a branching particle system interpreted as
a model of FKPP fronts. In this framework, it was shown that the
genealogy of the particle system converges to a Brownian Coalescent Point process on the timescale $N/\Sigma^2$.
The quantity $ N / \Sigma^2 $ is also close to the expected coalescence time of pairs of lineages in an individual-based model approximating the above SPDE studied in \cite{etheridge_genealogies_2022} (in which it was shown that the genealogy actually converges to a Kingman coalescent after a suitable time scaling).
The exact factor appearing in \cite{etheridge_genealogies_2022} (Theorem~1.2) differs in that it does not involve the term $ g(\tilde{h}(x)) $, and this comes from the choice we have made when decomposing the total population density $ v^N $ in neutral fractions $ u^{k,N} $ in \eqref{def:ukN_kpp}.
Generalising Theorem~\ref{th:main_result} to this infinite-dimensional setting would thus provide a rigorous argument for the convergence of genealogies inside such stochastic travelling waves to a Kingman coalescent.

We can determine the conditions under which $ \Sigma^2 < \infty $ by computing the exponential rate of decay of $ \tilde{h} $ at $ +\infty $.
Solving the linearised equation of the travelling wave near $ v = 0 $, we obtain that $ \tilde{h}(x) \propto e^{-\lambda x} $ as $ x \to +\infty $ where
\begin{align*}
	\lambda & = \frac{c}{2D} \left( 1 + \sqrt{1 - \frac{4 D r(0)}{c^2}} \right)
	= \frac{c}{2D} \left( 1 + \sqrt{1 - \left( \frac{c_0}{c} \right)^2} \right),
\end{align*}
using \eqref{def:c0}.
We obtain that $ \Sigma^2 < \infty $ as long as $ \frac{c}{D} < \frac{3 \lambda}{2} $, which holds if and only if
\begin{equation*}
	c > \frac{3}{2 \sqrt{2}} c_0.
\end{equation*}
This condition was obtained in \cite{birzu_fluctuations_2018} (Equation 7), and defines what the authors called \textit{fully pushed fronts}.
When $ c \in \left( c_0, \frac{3}{2\sqrt{2}} c_0 \right) $, the authors say that the front is \textit{semi-pushed}.
In that case, the genealogy no longer converges to a Kingman coalescent, and is expected to converge to a Beta coalescent instead, on a different time scale, see \cite{tourniaire2024branching,foutel2024convergence}.
As a result, we can expect that $ (\theta^{k,N}(t), t \geq 0) $ converges on the same time scale to a $ \Lambda $-Fleming-Viot process which is the moment dual of the corresponding Beta coalescent.

\section{Non-uniqueness of neutral fractions}
\label{sect:unicity}

It is easy to see that, even if the SDE \eqref{eq:SDE_tot_pop} possesses the $\infty$-decomposable property, the decomposition is not necessarily unique.
To illustrate this, consider the case $E= \{1,2\}$ and
\[
	b(v) =
	\begin{pmatrix}
		b_0 v_1 + b_{11} (v_1)^2 + b_{12} v_1v_2 \\
		b_0 v_2 + b_{22} (v_2)^2 + b_{21} v_1v_2
	\end{pmatrix},
\]
with $b_0$, $b_{ij} \in \mathbb{R}$ fixed constants, so that the deterministic component of the model corresponds
with a Lotka-Volterra model. For any parameters $\alpha, \beta > 0$, the drift $b$ can be decomposed as
\[
	b(v) = F_{\alpha,\beta}(v) v,
\]
where
\begin{equation}
	\label{eq:ex_F}
	F_{\alpha,\beta}(v_1,v_2) =
	\begin{pmatrix}
		b_0 + b_{11} v_1 + (b_{12} - \alpha) v_2 & \alpha v_1                           \\
		\beta v_2                                & b_0 + b_{22} v_2 + (b_{21}-\beta)v_1
	\end{pmatrix}.
\end{equation}

This defines a two-parameter family of decompositions corresponding to the same diffusion process. Importantly, the fixed point  $\tilde{h}$  is \emph{independent} of the choice of $\alpha$ and $\beta$, while the {right-eigenvector} $h\equiv h_{\alpha,\beta}$ of $F_{\alpha,\beta}(\tilde{h})$ \emph{does} depend on these parameters.
According to (\ref{eq:Ne}),
this leads to a definition of the {effective population size}
\[
	N_e = \frac{N}{\langle h_{\alpha,\beta},\, a(\tilde{h}) a^*(\tilde{h})\, h_{\alpha,\beta} \rangle},
\]
which depends on the choice of $F_{\alpha,\beta}$, but {not} on how the diffusion coefficient $a$ is decomposed.

Naturally, this raises the question: \emph{What is the most natural choice of $\alpha$ and $\beta$?} The answer depends on the underlying microscopic model.
To explore this, we consider the generic two-type model introduced in Section~\ref{sect:asexual-populations}, with $E = \{1, 2\}$, and microscopic transition rates $\rho_{\alpha,\beta}$ chosen such that
\[
	F_{\alpha,\beta}({\bf V})_{kl}
	= \sum_{{\bf n} \in \mathbb{N}_0^2} \rho_{\alpha,\beta}(k, {\bf V}, {\bf n}) \left( {\bf n}_l - \delta_{kl} \right), \quad k,l \in \{1, 2\}.
\]
With this specification, the discrete generator $G_{\alpha,\beta}^N f$ converges to a limiting generator $G f$, which is {independent} of the parameters $\alpha$ and $\beta$. However, an analogous analysis for the generator of the {neutral fractions} shows that, up to an error of order $o(1/N)$, the resulting generator coincides with that of the limiting neutral dynamics~(\ref{eq:SDE_fraction}) associated with $F_{\alpha,\beta}$.
Thus, while the macroscopic dynamics are independent of $\alpha$ and $\beta$, the neutral fraction dynamics retain a memory of this choice, highlighting the subtle role of microscopic structure in shaping neutral genealogies.
We refer the reader to \cite[Section 3.4]{etheridge2023looking} for an analogous discussion
for the lineage dynamics in the backward in time approach.

\section{Consistence and exchangeability}
\label{sect:neutral}

Let $(P_t^K; t\geq 0)$ be the semigroup associated to the neutral fraction of order $K$. From the $\infty$-decomposability property stated in Assumption~\ref{hyp:decomposability}, it is straightforward to check that the sequence $(P^K)_K$ is consistent and exchangeable in the following sense.
\begin{enumerate}
	\item {\bf Consistent}. Define
	      $$
		      \forall (z_i)_{i=1}^K \in (\R^E)^K, \quad \ m^K(z_1,\cdots, z_K) \ = \ (z_1,\cdots, z_{K-2}, z_{K-1}+z_{K})
	      $$
	      Then, for any $ f : (\R^E)^{(K-1)} \to \R $, $$P_t^{K}(f\circ m^K) \ = \ P_t^{K-1}(f)\circ m^K.$$
	\item {\bf Exchangeable}. For every permutation $p : (\R^E)^K \to (\R^E)^K$ and any $ f : (\R^E)^{K} \to \R $,
	      $$
		      P^{K}_t(f\circ p) \ =\ P^{K}_t(f)\circ p.
	      $$
\end{enumerate}
Note that those two properties were enforced by the semilinear structure of the
drift and covariance of our SDE in Assumption~\ref{hyp:decomposability}.

This raises the following question. Consider a sequence of semi-groups $(P^K)$ such that $P^1$ is the semi-group associated to an SDE of the form (\ref{eq:SDE_tot_pop}). If we assume that $(P^K)$ is consistent and exchangeable, what can be said about the algebraic structure of the coefficients in (\ref{eq:SDE_tot_pop})?
We defer this question to future work.

\section{Proof of Theorem~\ref{th:main_result}}
\label{sec:proof}

\subsection{Notation}
For $u\in\M$, we write $u^j\in \R^E$ for its $j$-th column and we denote
by $u_{xj}$ its entry corresponding to `site' $x$ in the $j$-th column
(or fraction).
We write $ (e_x)_{x\in E}$ for the canonical basis of $ \R^E $.
By an abuse of notation, we will use $\|\cdot \|$
to denote the Euclidean norm on $\R^E$, the Euclidian norm on $\R^K$, and
for the induced matrix norm on $\M$ and on $\R^{E\times E}$.
	{Throughout the paper, $M$ and $\gamma$ denote positive constants whose values
		may change from line to
		line. Numbered constants keep the same value throughout the text.}
To ease notation, we will write $\S$ for the linear map that sums the
columns of an element of $\M$, that is,
\begin{equation*}
	\S: \M \to \Rv, \quad u\mapsto \S(u)=\sum_ju^j,
\end{equation*}
and define
\begin{equation}
	\label{eq:defSk}
	\tS^k:  \M \to  \Rv\times \Rv, \quad
	u\mapsto \tS^k(u)=(\S(u),u^k).
\end{equation}
With this notation, the system of $|E|$-dimensional SDEs \eqref{eq:SDE_fraction} can then be reformulated
as a \textit{matrix-valued SDE}
\begin{equation}
	\label{eq:matrix_valued_SDE}
	dU^N_t=F(\S(U^N_t))U_t^N dt+\frac{1}{\sqrt{N}}\G(U^N_t)\cdot d\W_t,
\end{equation}
where $U_t=(U_t^1,...,U_t^K)$ refers to the composition matrix,
$\W_t=(W_t^1,...,W_t^K)$ is the noise matrix (the $(W^i)$ are those defined in
\eqref{eq:SDE_fraction}), $\G$ is the matrix indexed by pairs of indices
such that
\begin{eqnarray}
	\label{eq:def_G}
	\G(u)_{xj,yk}=\delta_{jk}\sigma(\tS^k(u))_{xy}, \quad x,y \in E, \,  \; j, k \in [K],
\end{eqnarray}
and the product $\cdot$ is defined by
\begin{equation}
	\label{eq:prod_matrix}
	[\G(u)\cdot d\W_t]_{xj}=\sum_{\underset{k\in[K]}{y\in E}}\G(u)_{xj,yk}(d\W_t)_{yk}
	= \sum_{y\in E}\sigma(\S(u),u^j)_{xy}(d\W_t^j)_y
	=(\sigma(\S(u),u^j) d\W_t^j)_x.
\end{equation}
(In words, the above quantity corresponds to the noise component affecting
the $j$-th family at site $x$.)
We write $\lf\cdot,\cdot\rf$ for the associated scalar product in $\R^{E\times K}$,
	that is, the Frobenius inner product.

Under Assumption~\ref{hyp:decomposability},
if $U^N$ denotes the unique weak solution to \eqref{eq:matrix_valued_SDE},
$\S(U^N)$ is the unique weak solution of~\eqref{eq:SDE_tot_pop}, and corresponds to the total population size of the system.

The limiting objects obtained in Theorem \ref{th:main_result} can be rewritten in
terms of the limits of deterministic flows that we now introduce.
Consider the deterministic flow associated to the SDE \eqref{eq:SDE_tot_pop} defined  as
\begin{align}\label{eq: detflow}
	\phi: \R_+\times\Rv\to \Rv, \quad \phi(t,v_0) := v_t,
\end{align}
where $v_t$ is the unique solution of~$\dot{v}_t=b(v_t)$ with initial
condition $v_0$. Note that, under Assumption~\ref{A2},
for $v_0\in\B$,
\begin{equation}
	\label{eq:lim_flow}
	\lim_{t\to\infty}\phi(t,v_0)=\tilde h.
\end{equation}
We further define
\begin{align}
	\label{eq:detFlow}
	\varphi:\R_+\times\Rv\times \Rv \to \Rv,
	\quad \varphi(t,v_0,w_0)=w_t,
\end{align}
where $w_t$ is the unique solution of the linear ODE
\begin{equation}
	\label{eq:flot_u}
	\frac{d}{dt}w_t=F(\phi(t,v_0))w_t
\end{equation}
with initial condition $w_0$, and let $\Bs$
be the basin of attraction of the manifold $\Gamma^K$, that is,
\begin{equation}
	\label{eq:BSt}
	\Bs=\{u_0\in \Mp:  \S(u_0)\in\mathcal{B}\}.
\end{equation}
For $j\in [K]$ and $u_0\in\Bs$, we write
\begin{equation}
	\label{eq:def_Katz}
	\pi^j(u_0)=\lim_{t\to\infty}
	\varphi(t,\S(u_0),u_0^j).
\end{equation}

\begin{definition}
	For a given composition matrix $u_0\in \Bs$, the matrix $(\pi^1(u_0),...,\pi^K(u_0))$ is referred to as the Katzenberger projection of $u_0$.
\end{definition}
We conclude this section by recalling that under Assumptions~\ref{A1} and \ref{A3},
\begin{equation}
	\label{eq:PR_dec}
	\forall t\in[0,\infty), \quad
	e^{tF(\tilde h)}=P+R_t,
\end{equation}
where $P$ is the Perron-Frobenius projector defined by
\begin{equation}
	\forall w\in \Rv,  \quad Pw=\l w,h\r\tilde h,
	\label{eq:pf_proj}
\end{equation}
and that there exist constants $M_1,\gamma_1>0$ such that
\[
	\forall t\in[0,\infty), \quad \|R_t\|\leq M_1e^{-\gamma_1 t}.
\]
On the other hand, Assumption~\ref{A2} ensures the existence of constants $M_2,\gamma_2>0$ such that,
\begin{equation}
	\label{eq:J_dec}
	\forall t\in [0,\infty), \quad \|e^{tJ}\|\leq M_2e^{-\gamma_2 t}.
\end{equation}
\begin{definition}
	\label{def:pf}
	{For a given composition matrix $u_0\in \Bs$, the matrix $(P(u_0^1),...,P(u_0^K))$ is referred to as the Perron-Frobenius projection of $u_0$.}
\end{definition}

\subsection{Two approaches to Katzenberger’s method for $\infty$-decomposable SDEs}

In the present work, we provide two distinct proofs of Theorem~\ref{th:main_result} that we now explain
and motivate. Both proofs rely on the fact that, under our decomposability condition,
Katzenberger projections can be expressed in terms of \textit{reproductive values}
(that is, the respective mean long-term contributions of
type-$x$ individuals).
In standard branching population models, such as multitype critical Galton-Watson processes,
the reproductive values are derived from the Perron-Frobenius decomposition of the
semigroup; and are given by the left eigenvector of the
mean offsrping matrix (see e.g.,~\cite{athreya2012branching}). Under
the decomposability assumption (Assumption~\ref{hyp:decomposability}) and the
Perron-Frobenius type condition (Assumption~\ref{A3}), we have a similar
interpretation of the eigenvector $h$, as we now describe.

Consider the deterministic component of~\eqref{eq:SDE_tot_pop},
$
	\dot{v}_t = F(v_t) v_t.
$
If we freeze the mean matrix at equilibrium and assume that $v_t$ only deviates
from equilibrium by a small perturbation $w_t$ then we obtain the linearised problem
\[
	\frac{d}{dt} w_t = F(\tilde h) w_t
\]
with initial condition $w_0$. This system can be viewed as describing the  evolution
of small families at the deterministic equilibrium, where their growth follows
a branching process. We then interpret $h$, the left
eigenvector of $F(\tilde h)$, as the reproductive value at the deterministic equilibrium.
Indeed, in the above linearised system, the long-term contribution
of such a family is given by the Perron-Frobenius projection~\eqref{eq:pf_proj}; that is, a scalar product against the reproductive value:
\[
	\lim_{t\to\infty} w_t = \l h,w_0\r\tilde h.
\]
In general, the deterministic component of the dynamics of a subpopulation depends on the total population size:
\[
	\frac{d}{dt} u_t^k = F(S(u_t)) u_t^k,
\]
with initial condition $u_0$, $u_0^k$. Now the size of the subpopulation in the limit is given by the Katzenberger projection of $u_0$~\eqref{eq:def_Katz}. In Proposition~\ref{prop:scalar_form} we show that, analogously to the Perron-Frobenius projection, the Katzenberger projection is represented by a scalar product, so that in this general case we have
\[
	\lim_{t\to\infty}u_t^k = \l H(S(u_0)),u_0^k\r\tilde h,
\]
where $H:\R^E\to\R^E$ is interpreted as the reproductive value (which depends on the total population size) in a regulated population. Proposition~\ref{prop:scalar_form} will also prove that on the invariant manifold $\Gamma^K$, the Perron-Frobenius and Katzenberger reproductive values and therefore the corresponding projections agree.

	{\bf First approach {(via Katzenberger's method)}.}
The first proof relies directly on a general result of Katzenberger~\cite{katzenberger1990solutions} restated in Appendix~\ref{sec:katzenberger}.
This result
states that the process \( (U^N_{{tN}/{\Sigma^2}},\ t \in [\varepsilon,T]) \) converges in distribution to a process taking values in \( \Gamma^K \), and characterises the limit as the solution to a stochastic differential equation (SDE) involving the first and second derivatives of the Katzenberger projections \( \pi^j \).
We show that the specific structure of our infinitely decomposable SDE allows simplifications in the limiting equation, leading to the recovery of the multi-dimensional Wright--Fisher diffusion. This argument is developed in Section~\ref{subsec:main_proof}.

While broadly applicable, the arguments underlying Katzenberger's theorem are technically involved and difficult to interpret, making their extension to infinite-dimensional settings, such as the PDEs discussed earlier in Section \ref{sec:IBM}, unclear.

\textbf{Second approach  {(using Perron-Frobenius reproductive values)}.}
To address this, we present a second proof which, although less general, is tailored to the $\infty$-decomposable setting.
While
still relying on the crucial cancellations in Itô’s formula that underpin Katzenberger’s analysis, this proof is inspired by the geometric interpretation introduced by Parsons and Rogers~\cite{parsons2017dimension}.

In this proof, we rely on the fact that, in a small neighbourhood of the stable manifold,
the Katzenberger reproductive value $H(\cdot)$ is well approximated by the Perron–Frobenius
reproductive value $h$.
This approach provides sharper quantitative bounds on convergence
to the stable manifold, bounds that do not rely on Lyapunov-type arguments as in~\cite{katzenberger1990solutions} and~\cite{parsons2017dimension}.

This alternative approach, which is presented in Section~\ref{sec:2nd_proof},
is not a mere exercise of simplification,
but develops a tractable geometric methodology that lays the groundwork for
extending the analysis to the infinite-dimensional setting,
both in terms of the number of families (Fleming--Viot processes) and the
complexity of the type space~$E$. In particular, it paves the way for applying our results to the PDEs introduced in Sections~\ref{sec:gut} and \ref{sect:F-KPP}.

\subsection{Proof of Theorem~\ref{th:main_result} via Katzenberger's approach} \label{subsec:main_proof}

To apply Katzenberger's result (Theorem~\ref{thm:katzenberger} below), we identify
$\Mp$
with $ \R_{\geq 0}^d $ with $ d = | E |\times K  $.
We nonetheless continue to index vectors by $ \lbrace xk, x \in E, k \in [K] \rbrace $,
and, for $ u \in \R_{\geq 0}^d $, we let $ u^k $ denote the element of $ \Rdp $
corresponding to the $ k $-th column of the corresponding element of $\Mp$.
The $\cdot$~product  can then be identified with the usual matrix multiplication in
$\R^{d\times d}$, and $ \mathbf{G}(u) $ can be seen as an element of $ \R^{d\times d} $.
For a twice continuously differentiable real-valued function
$ f : U \subset \R^d \to \R $, we let $ \nabla f : U \to \R^d $ (resp.~$ \nabla^2 f : U \to \R^{d\times d} $) denote its gradient (resp.~Hessian).

Before proving Theorem~\ref{th:main_result}, we establish several properties of
the Katzenberger projections in the infinitely decomposable setting.
In particular, the identities satified by the first and second order derivatives of the Katzenberger projections
derived in Lemma~\ref{lemma:katzenberger_projections}
are essential for obtaining the limiting Wright--Fisher diffusion.

\begin{lem} \label{lemma:katzenberger_projections}
	For all $ u \in \mathcal{B}^* $, $ \pi^k(u) $ can be written as
	\begin{equation} \label{def_theta}
		\pi^k(u) = \theta^k(u) \tilde{h},\qquad \theta^k(u) \in \R_{\!\scriptscriptstyle\geq 0}.
	\end{equation}
	For each $k\in [K]$, the map $ \theta^k : \mathcal{B}^* \to \R $ is thrice continuously differentiable.

	For $u\in \Bs$ and $k,k'\in[K]$, define
	\begin{equation}
		\label{eq:def_Cov}
		{\C}_{k,k'}(u):=\lf   \mathbf{G}(u) \cdot \nabla \theta^k(u), \mathbf{G}(u) \cdot \nabla \theta^{k'}(u) \rf
		\quad \text{and}
		\quad
		\C_{k,k'}^{WF}(u):=(\mathbf{1}_{k=k'} - \theta^{k'}(u)) \theta^k(u).
	\end{equation}
	Then,  for all $ u \in \Gamma^K$ and $k,k'\in[K]$,
	\begin{equation} \label{covariance_theta_k}
		\C_{k,k'}(u)= \Sigma^2  \C_{k,k'}^{WF}(u),
	\end{equation}
	where $\Sigma^2$ is given by~\eqref{eq:Sigma-ancestral},
	and
	\begin{equation} \label{Trace_on_Gamma}
		\mathrm{Tr}\left[ \mathbf{G}(u)^* \cdot \nabla^2 \theta^k(u) \cdot \mathbf{G}(u) \right] = 0.
	\end{equation}
\end{lem}

We prove Lemma~\ref{lemma:katzenberger_projections} in Section~\ref{sec:proof_lemma5} below.
For now, let us show how to deduce Theorem~\ref{th:main_result} from this.

\begin{proof}[Proof of Theorem~\ref{th:main_result}]
	To prove Theorem~\ref{th:main_result}, we consider the open set $ \mathcal{U} = \mathcal{B}^* $ and the flow $ \mathbf{F} $ defined on $ \mathcal{U} $ as
	\begin{equation*}
		\mathbf{F}(u)_{xk} := [F(\mathcal{S}(u)) u^k ]_x.
	\end{equation*}
	It then follows from the definition of $ \mathcal{B}^* $ that
	\begin{equation*}
		\Gamma = \mathbf{F}^{-1}(0) = \lbrace u \in \mathcal{B}^* : \mathcal{S}(u) = \tilde{h} \rbrace = \Gamma^K,
	\end{equation*}
	which is a submanifold of $ \mathcal{U} $ of dimension $ m = K-1 $\footnote{{Note that this manifold is not $\mathcal{C}^2$ at its corners. Katzenberger’s approach addresses this technical issue by stopping the diffusion upon exiting a small ball centered at an interior point of the manifold}.}.
	The flow $ \psi(\cdot,t) $ defined in \eqref{def_flow_katz} is then such that
	\begin{equation*}
		\psi(u,t)^k = \varphi(t,\mathcal{S}(u), u^k),
	\end{equation*}
	and the limit $ \Phi : \mathcal{B}^* \to \Gamma $ is given by
	\begin{equation*}
		\Phi(u) = (\pi^1(u), \ldots, \pi^K(u)) = (\theta^1(u) \tilde{h}, \ldots, \theta^K(u) \tilde{h}).
	\end{equation*}
	The matrix-valued SDE \eqref{eq:matrix_valued_SDE} can then be written under the form \eqref{katz:sde} with $ X_N(t) := U^{N}_{Nt} $, $ Z_N = \mathbf{W} $, $ A_N(t) = Nt $ and $ \mathbf{G}_N = \mathbf{G} $.
	All the assumptions on $ Z_N $ and $ A_N $ are then trivially satisfied.

	We next check the assumption on the eigenvalues of the Jacobian matrix of $ \mathbf{F} $ on the submanifold $ \Gamma^K $.
	This follows by noting that
	\begin{equation*}
		\mathcal{S}(\mathbf{F}(u)) = F(\mathcal{S}(u)) \mathcal{S}(u) = b(\mathcal{S}(u)).
	\end{equation*}
	By the linearity of $ \mathcal{S} $, we see that, if $ J_{\mathbf{F}}(u) $
	denotes the Jacobian matrix of $ \mathbf{F} $ at $ u \in \Gamma^K $,
	\begin{equation*}
		\mathcal{S}(J_{\mathbf{F}}(u) \cdot X) = J_b(\tilde{h}) \mathcal{S}(X).
	\end{equation*}
	{It follows that the spectrum of $ J_{\mathbf{F}}(u) $ is included in
	that of $(J_b(\tilde{h})) \cup \lbrace 0 \rbrace $, and that the eigenvectors associated to 0 are such that $ \mathcal{S}(X) = 0 $.
	If we now take $ u \in \Gamma $ and $ X \in \R^d $ such that $ \mathcal{S}(X) = 0 $, we see that
	\begin{equation*}
		[\mathbf{F}(u + X) - \mathbf{F}(u)]^k = F(\tilde{h}) X^k.
	\end{equation*}
	Hence, for such $ u $ and $ X $,
	\begin{equation*}
		[J_{\mathbf{F}}(u) \cdot X]^k = F(\tilde{h}) X^k.
	\end{equation*}
	By Assumption~\ref{A3}, we conclude that $ J_{\mathbf{F}}(u) \cdot X = 0 $ if and only if $ \mathcal{S}(X) = 0 $ and $ X^k $ is colinear to $ \tilde{h} $ for all $ k \in [K] $.
	The dimension of the resulting eigenspace (associated to the eigenvalue 0) is thus exactly $ m = K-1 $.}
	It then follows from Assumption~\ref{A2} that the remaining $ d-m $ eigenvalues all have negative real parts.

	The regularity of $ \Phi $ on $ \mathcal{U}_\Gamma = \mathcal{B}^* $ follows from Lemma~\ref{lemma:katzenberger_projections}.
	We then consider a compact set $ K \subset \mathcal{B}^* $ such that $ \Gamma^K \subset \mathring{K} $ (this is possible by Assumption~\ref{A2}).
	Let us then define $ Y_N $ and $ \mu_N(K) $ as in \eqref{def:Yn}.
	Theorem~\ref{thm:katzenberger} then yields that the sequence $ (Y_N^{\mu_N(K)}, Z_N^{\mu_N(K)}, \mu_N(K)) $ is tight in $ D([0,T], \R_+^d \times \R^d) \times [0,\infty] $  and that the limit of any converging subsequence satisfies \eqref{katz_limiting_sde}.
	By the remark following Theorem~\ref{thm:katzenberger}, we also obtain that $ (X_N^{\mu_N(K)}, Z_N^{\mu_N(K)}, \mu_N(K)) $ is tight in $ D([\varepsilon,T], \R_+^d \times \R^d) \times [0,\infty] $ and that the limit of any converging subsequence also satisfies \eqref{katz_limiting_sde}.
	By the choice of $ K $ and \eqref{mu_liminf}, $ \mu = +\infty $ almost surely since $ Y(t) \in \Gamma $ for all $ t \geq 0 $.
	As a result, $ (X_N^{\mu_N(K)}, Z_N^{\mu_N(K)}, \mu_N(K)) $ also converges in distribution along any converging subsequence of $ (Y_N^{\mu_N(K)}, Z_N^{\mu_N(K)}, \mu_N(K)) $.

	To conclude, we compute the terms appearing in \eqref{katz_limiting_sde}.
	Thus let $ (Y,Z,\mu) $ denote the limit of a converging subsequence of $ (Y_N, Z_N, \mu_N)$.
	First note that, since $ Z_N = \mathbf{W} $ for all $ N $, then $ Z $ is also distributed as standard Brownian motion.
	Let $ (\theta^k(t), k \in [K]) $ be such that
	\begin{equation*}
		Y(t) = (\theta^1(t) \tilde{h}, \ldots, \theta^K(t) \tilde{h}).
	\end{equation*}
	Then \eqref{katz_limiting_sde} translates to
	\begin{equation} \label{sde_theta_k}
		\theta^k(t) = \theta^k(0) + \int_{0}^{t} \lf \nabla \theta^k(Y(s)) , \mathbf{G}(Y(s))  \cdot  d\mathbf{W}_s \rf+ \frac{1}{2} \int_{0}^{t} \mathrm{Tr} \left[ \mathbf{G}(Y(s))^* \cdot \nabla^2 \theta^k(Y(s)) \cdot \mathbf{G}(Y(s)) \right] ds.
	\end{equation}
	Since $ Y(s) \in \Gamma^K $, \eqref{Trace_on_Gamma} in Lemma~\ref{lemma:katzenberger_projections} implies that the second term on the right of \eqref{sde_theta_k} is equal to zero.
	We thus obtain that $ (\theta^1(t), \ldots, \theta^K(t))_{t \geq 0} $ is a continuous martingale (with respect to the filtration associated to $ (\mathbf{W}_t, t \geq 0) $) with quadratic variation process
	\begin{equation*}
		\left[ \theta^k, \theta^{k'} \right]_t = \int_{0}^{t} \nabla \theta^k(Y(s))^* \cdot \mathbf{G}(Y(s))^* \cdot \mathbf{G}(Y(s)) \cdot \nabla \theta^{k'}(Y(s)) ds.
	\end{equation*}
	Using \eqref{covariance_theta_k} from Lemma~\ref{lemma:katzenberger_projections} and the fact that $ Y(s) \in \Gamma^K $, we obtain
	\begin{equation*}
		\left[ \theta^k, \theta^{k'} \right]_t = \Sigma^2 \int_{0}^{t} \left( \ind_{k,k'} - \theta^{k'}(s) \right) \theta^k(s) ds.
	\end{equation*}
	We recover the covariance structure of a $ (K-1) $-dimensional Wright-Fisher diffusion, which concludes the proof.
\end{proof}

The next sections are devoted to the proof of Lemma~\ref{lemma:katzenberger_projections}.

\subsubsection{The scalar-product form}

The goal of this section is to prove that, under the $\infty$-decomposability
condition, Katzenberger projections exhibit a tractable algebraic structure.
\begin{prop}[Scalar-product form]
	\label{prop:scalar_form}
	There exists a map $H:\Rv\to\Rv$ such that, for all $u_0\in \Bs$,
	\begin{equation*}
		\pi^k(u_0)
		= \langle H(\S(u_0)), u_0^k \rangle\, \tilde{h}.
	\end{equation*}
	In particular, $H(\tilde{h}) = h$, and for all $v_0 \in \B$,
	\begin{equation} \label{Hv_v=1}
		\langle H(v_0), v_0 \rangle = 1.
	\end{equation}
\end{prop}
Intuitively, $H(v_0)_x$ quantifies the reproductive value of type-$x$ individuals
given the population configuration $v_0$. The second part of the result shows that,
on $\Gamma^K$, the Perron-Frobenius and the Katzenberger reproductive
values, and consequently the corresponding projections, coincide. Moreover,
\eqref{Hv_v=1} indicates that the dynamics encoded by \eqref{eq:SDE_tot_pop}
is \textit{critical}.

For simplicity, we assume that the set $\B$ in
Assumption~\ref{A2} is relatively compact.
The results below still hold for more general sets, by
working with compact subsets of $\B$.
Under this assumption, the Hartman-Grobman theorem
provides the existence of two
constants $M_3,\gamma_3>0$ such that, for all $ v_0\in \B$ and all $t\in[0,\infty)$,
\begin{equation}
	\label{eq:HB1}
	\|\phi(t,v_0)-\tilde h\|\leq M_3e^{-\gamma_3 t}{\|v_0-\tilde h\|}.
\end{equation}
Combining this bound with Assumption~\ref{A:regularity}, we get the existence of
a constant $M_4>0$ such that, for all $t\in[0,\infty)$ and all $v_0\in \B$,
\begin{equation}
	\label{eq:bound_F}
	\|F(\phi(t,v_0))-F(\tilde{h})\|\leq
	M_4e^{-\gamma_3 t}\,\|v_0-\tilde h\|.
\end{equation}

\begin{lem}
	\label{lem:theta}
	Recall the definition of the flow $\varphi$ from \eqref{eq:detFlow}.
	For all $v_0\in \B$ and all $w_0\in \Rv$,
	\begin{equation*}
		\lim_{t\to\infty}\varphi(t,v_0,w_0)=
		\theta(v_0,w_0)\tilde h,
	\end{equation*}
	with
	\begin{equation}
		\label{eq:def_ell}
		\theta(v_0,w_0)=\l w_0+ \int_0^\infty (F(\phi(s,v_0)-F(\tilde h)))
		\varphi(s,v_0,w_0)ds,h\r, \quad  v_0\in \B, \; w_0\in \Rv.
	\end{equation}
\end{lem}
\begin{proof}
	By definition of the flow $\varphi$, we have
	\begin{equation}
		\label{eq:frozen_dynamic}
		\frac{d\varphi}{dt}(t,v_0,w_0)=F(\phi(t,v_0))\varphi(t,v_0,w_0)
		=F(\tilde h)\varphi(t,v_0,w_0)
		+(F(\phi(t,v_0)-F(\tilde h)))\varphi(t,v_0,w_0).
	\end{equation}
	By Duhamel's formula, the solution of this ODE can be written as
	\begin{equation}
		\label{eq:duhamel1}
		\varphi(t,v_0,w_0)=e^{tF(\tilde h)}w_0+\int_{0}^{t}e^{(t-s)F(\tilde h)}
		(F(\phi(s,v_0)-F(\tilde h)))\varphi(s,v_0,w_0)ds.
	\end{equation}
	We then see from \eqref{eq:bound_F} and the triangle inequality that
	\begin{equation}
		\label{eq:duhamel_phi}
		\|\varphi(t,v_0,w_0)\|\leq\|w_0\|+M\int_{0}^{t}e^{-\gamma s}\|\varphi(s,v_0,w_0)\|ds.
	\end{equation}
	Grönwall's inequality thus yields
	\begin{equation}
		\label{eq:bound_vphi0}
		\,\forall t\in[0,\infty),\quad \|\varphi(t, v_0 ,w_0)\|\leq M\|w_0\|,
	\end{equation}
	and this bound holds uniformly in $v_0\in\B$.

	Next, recall from \eqref{eq:PR_dec} that
	\begin{equation*}
		e^{tF(\tilde h)}w_0\xrightarrow[]{t\to\infty}P(w_0)= \l w_0,h\r \tilde h.
	\end{equation*}
	We then apply the same decomposition \eqref{eq:PR_dec} to split the integral in \eqref{eq:duhamel1}
	into two terms, and show that, for all $t\in[0,\infty)$,
	\begin{equation*}
		I_1(t):=P\int_0^t(F(\phi(s,v_0)-F(\tilde h)))\varphi(s,v_0,w_0)ds \,
	\end{equation*}
	is colinear to $\tilde h$
	(which is clear by the definition of the projector $P$) and that
	\begin{equation}
		\label{eq:I2}
		I_2(t):=	\int_0^tR_{t-s}\left[(F(\phi(s,v_0)-F(\tilde h)))\varphi(s,v_0,w_0)\right]ds
		\xrightarrow[]{t\to\infty} 0.
	\end{equation}
	To show \eqref{eq:I2}, we observe that,
	by the triangle inequality, \eqref{eq:PR_dec}, \eqref{eq:bound_F} and \eqref{eq:bound_vphi0},  we have
	\[
		\| I_2(t) \| \leq
		M\left(\int_0^{t_0} \|R_{t-s}\|ds+\int_{t_0}^t \|F(\phi(s,v_0))-F(\tilde h)\|ds\right)
		\quad \text{for any } t_0>0,
	\]
	which converges to $0$ when we first let $t\to\infty$ and then $t_0\to\infty$.
	We conclude the proof of the lemma by noting that the integral in $I_1(t)$
	is absolutely convergent by \eqref{eq:bound_vphi0} and \eqref{eq:bound_F}.
\end{proof}

\begin{proof}[Proof of Proposition~\ref{prop:scalar_form}]
	First, we note that, for fixed $v_0 \in \B$, the map $w_0 \mapsto \theta(v_0, w_0)$
	defines a linear form on $\Rv$: this is clear
	from \eqref{eq:def_ell} since the flow $\varphi(t,v_0,\cdot)$ is linear.
	This implies that, for all $v_0\in \B$,
	\begin{equation*}
		\exists ! H(v_0): \quad \forall w_0\in \Rv, \quad  \theta(v_0,w_0)=\langle H(v_0), w_0 \rangle.
	\end{equation*}
	Putting this together with \Cref{lem:theta} and \eqref{eq:def_Katz} yield
	the first part of the result.

	Let us now prove that $H(\tilde h)=h$ and that $\theta(v_0,v_0)=1$ for all $v_0\in \B$.
	For $v_0=\tilde h$, we see that
	\[\theta(v_0,w_0)=\l H(v_0), w_0 \r=\l w_0, h\r, \quad \forall w_0\in \Rv
		.\]
	This shows that $H(\tilde h )=h$. The second point is a direct consequence of \eqref{eq:lim_flow}.

\end{proof}

\begin{nota} In the remainder of the section, we write $\theta$ for the map
	defined by
	\begin{equation}
		\label{eq:def_theta}
		(v_0,w_0)\in \B \times \Rv
		\mapsto \theta(v_0,w_0)=\langle H(v_0),w_0 \rangle \in \R.
	\end{equation}
	With the notation introduced in Lemma~\ref{lem:theta} and \eqref{eq:defSk}, we have
	\begin{equation}
		\label{eq:pi_k_comp}
		\pi^k(u_0)=\theta \circ \tS^k(u_0)\tilde h=:\theta^k(u_0)\tilde h, \quad u_0\in \Bs.
	\end{equation}
\end{nota}

The following result is a direct consequence of Proposition~\ref{prop:scalar_form}.
\begin{cor}
	\label{cor:manifold}
	For all $u_0\in \Bs$,
	\begin{equation*}
		(\theta^k(u_0))_{k \in [K]} \otimes \tilde{h} \in \Gamma^K.
	\end{equation*}
\end{cor}
\begin{proof}
	By Assumption~\ref{A:regularity}-3, for all $u_0 \in \Bs$ and $k \in [K]$, we have $\theta(\S(u_0), u_0^k) \geq 0$. Furthermore, since $\theta$ is linear in its second argument and by \eqref{Hv_v=1}, it follows that for all $u_0 \in \Bs$,
	\[
		\sum_{k} \theta(\S(u_0), u_0^k) = \theta(\S(u_0), \S(u_0)) = \langle H(\S(u_0)), \S(u_0) \rangle = 1.
	\]
\end{proof}

\subsubsection{{Regularity of Katzenberger projections}}
\begin{lem}
	\label{lem:reg_katzenberger}
	{For all $w_0 \in \mathbb{R}^E$, the map $v_0 \mapsto \theta(v_0, w_0)$ is
		three times continuously differentiable on~$\mathcal{B}$.}
\end{lem}
The proof of this lemma is postponed to Appendix~\ref{sec:proof_regularity}.

\begin{cor}
	\label{cor:H_smooth}
	The map $H:\B\to \R$ is three times continuously differentiable.
	In particular, there exists a constant $M>0$ such that,
	\[ \forall v_0\in \B, \quad \|H(v_0)-h\|\leq M\|v_0-\tilde h\|.\]
\end{cor}
\begin{proof}
	The results stems from the previous lemma together with the fact that
	\[
		\forall v_0\in\B, \; \forall x\in E, \quad [H(v_0)]_x=\theta(v_0,e_x).
	\]
\end{proof}

The following result will allow us to apply Ito's formula to $ (\theta^k(U^N_t), t \geq 0) $.
Note that we use the notation for Fréchet derivatives for $ H : \B \to \R^E $, but since
$ \theta^k : \M \to \R $, we use standard $ \nabla \theta^k $ and $ \nabla^2 \theta^k $
to denote its gradient and its Hessian matrix, respectively.

\begin{lem} \label{lemma:grad_theta_k}
	The map $ \theta^k : \Bs \to \R $ is twice continuously differentiable and, for all $ u \in \Bs $,
	\begin{equation} \label{expr_nabla_thetak}
		\nabla \theta^k(u)_{xj} = \langle DH(\S(u))(e_x), u^k - \ind_{j=k} \S(u) \rangle, \quad \forall\, x \in E,\; j \in [K],
	\end{equation}
	and, for $ x, y \in E $ and $ j, \ell \in [K] $,
	\begin{equation} \label{expr_nabla2_thetak}
		\nabla^2 \theta^k(u)_{xj, y\ell} = \langle D^2 H(\S(u))(e_x, e_y), u^k \rangle + \ind_{j=k} \langle DH(\S(u))(e_y), e_x \rangle + \ind_{k=\ell} \langle DH(\S(u))(e_x), e_y \rangle.
	\end{equation}
	Moreover, for all $ u \in \Bs $,
	\begin{equation} \label{katzenberger_cancellation}
		\lf\nabla \theta^k(u) , F(\S(u)) u \rf = \sum_{x,j} \nabla \theta^k(u)_{xj} (F(\S(u))u^j)_x = 0.
	\end{equation}
\end{lem}

\begin{proof}
	From Proposition~\ref{prop:scalar_form},
	\begin{equation*}
		\theta^k(u) = \langle H(\S(u)), u^k \rangle.
	\end{equation*}
	As a result, for any $ \delta \in \M $,
	\begin{equation*}
		\lf\nabla \theta^k(u), \delta \rf= \langle D[H \circ \S](u)(\delta), u^k \rangle + \langle H(\S(u)), \delta^k \rangle.
	\end{equation*}
	By the chain rule, $ D[H \circ \S](u)(\delta) = D H (\S(u))(D\S(u)(\delta)) $ and since $ \S $ is linear, $ D\S(u)(\delta) = \S(\delta) $.
	Hence
	\begin{equation*}
		\lf \nabla \theta^k(u) , \delta \rf = \langle DH(\S(u))(\S(\delta)), u^k \rangle + \langle H(\S(u)), \delta^k \rangle.
	\end{equation*}
	Moreover, differentiating \eqref{Hv_v=1} with respect to $ v $ yields, for any $ \Delta \in \R^E $ and $ v \in \B $,
	\begin{equation} \label{conservation_1}
		\langle DH(v)(\Delta), v \rangle + \langle H(v), \Delta \rangle = 0.
	\end{equation}
	Applying this with $ v = \S(u) $ and $ \Delta = \delta^k $, we obtain
	\begin{equation*}
		\lf\nabla \theta^k(u) , \delta \rf = \langle D H(\S(u))(\S(\delta)), u^k \rangle - \langle DH(\S(u))(\delta^k), \S(u) \rangle.
	\end{equation*}
	This proves \eqref{expr_nabla_thetak}.
	Differentiating again with respect to $ u $, we obtain, for $ \delta_1, \delta_2 \in \M $,
	\begin{multline*}
		\lf  \nabla^2 \theta^k(u)  \cdot \delta_2 ,\delta_1\rf= \langle D^2 H(\S(u))(\S(\delta_1), \S(\delta_2)), u^k \rangle + \langle DH(\S(u))(\S(\delta_1)), \delta_2^k \rangle \\ - \langle D^2 H(\S(u))(\delta_1^k, \S(\delta_2)), \S(u) \rangle - \langle D H(\S(u))(\delta_1^k), \S(\delta_2) \rangle.
	\end{multline*}
	However, differentiating \eqref{conservation_1} with respect to $ v $, we obtain, for any $ \Delta_1, \Delta_2 \in \R^E $,
	\begin{equation} \label{conservation_2}
		\langle D^2 H(v)(\Delta_1, \Delta_2), v \rangle + \langle D H(v)(\Delta_1), \Delta_2 \rangle + \langle D H(v)(\Delta_2), \Delta_1 \rangle = 0.
	\end{equation}
	Applying this with $ v = \S(u) $, $\Delta_1 = \delta_1^k $ and $ \Delta_2 = \S(\delta_2) $, we obtain
	\begin{equation*}
		\lf \nabla^2 \theta^k(u) \cdot  \delta_2, \delta_1\rf= \langle D^2 H(\S(u))(\S(\delta_1), \S(\delta_2)), u^k \rangle + \langle DH(\S(u))(\S(\delta_1)), \delta_2^k \rangle + \langle D H(\S(u))(\S(\delta_2)), \delta_1^k \rangle.
	\end{equation*}
	This yields \eqref{expr_nabla2_thetak}.
	To prove \eqref{katzenberger_cancellation}, for $ u \in \Bs $, let $ (\Phi_t(u), t \geq 0) $ denote the solution to
	\begin{equation}
		\label{eq:matrix_flow}
		\begin{aligned}
			\frac{d}{dt} \Phi_t(u) & = F(\S(\Phi_t(u))) \Phi_t(u), \quad \Phi_0(u)=u.
		\end{aligned}
	\end{equation}
	Then, noting that $ \theta^k(\Phi_t(u)) = \theta^k(u) $ for all $ t \geq 0 $, we obtain
	\begin{equation*}
		\frac{d}{dt} \theta^k(\Phi_t(u)) = \lf\nabla \theta^k(\Phi_t(u)) , (F(\S(\Phi_t(u))) \Phi_t(u)) \rf= 0.
	\end{equation*}
	Evaluating this expression for $ t = 0 $ yields \eqref{katzenberger_cancellation}.
\end{proof}

\subsubsection{Proof of Lemma~\ref{lemma:katzenberger_projections}}
\label{sec:proof_lemma5}
The first statement of the lemma was established in Proposition~\ref{prop:scalar_form}
and Corollary~\ref{cor:manifold}. The second statement about differentiability follows from Lemma~\ref{lem:reg_katzenberger} and~\eqref{eq:pi_k_comp}.
We now turn to the proof of \eqref{covariance_theta_k}-\eqref{Trace_on_Gamma}.

First, note that
\begin{equation*}
	\mathrm{Tr}\left[\G(u)^* \cdot \nabla^2\theta^k(u) \cdot \G(u) \right] =
	\sum_{\underset{j \in [K]}{z\in E}} \sum_{\underset{n \in [K]}{x \in E}}\sum_{\underset{m \in [K]}{y \in E}}
	\nabla^2 \theta^k(u)_{ym,xn}
	\G(u)_{xn,zj}\G(u)_{ym,zj}.
\end{equation*}
By the definition of $ \G(u) $ in \eqref{eq:def_G},
\begin{align*}
	\sum_{\underset{j \in [K]}{z\in E}} \G(u)_{xn,zj}\G(u)_{ym,zj} & = \ind_{n = m} \sum_{z\in e} \sigma(\S(u), u^n)_{xz} \sigma(\S(u), u^m)_{yz} = \ind_{n = m} \sigma \sigma^*(\S(u), u^n)_{xy}.
\end{align*}
Using \eqref{def_C}, we have
\begin{equation}  \label{GG*}
	\sum_{\underset{j \in [K]}{z\in E}} \G(u)_{xn,zj}\G(u)_{ym,zj} = \ind_{n = m} \left( C(\S(u)) u^n \right)_{xy}.
\end{equation}
As a result,
\begin{equation*}
	\mathrm{Tr}\left[\G(u)^* \cdot \nabla^2\theta^k(u) \cdot \G(u) \right] = \sum_{x \in E} \sum_{y \in E} \sum_{n \in [K]} \nabla^2 \theta^k(u)_{yn,xn} \left( C(\S(u)) u^n \right)_{xy}.
\end{equation*}
Plugging \eqref{expr_nabla2_thetak} in this expression yields
\begin{multline*}
	\mathrm{Tr}\left[\G(u)^* \cdot \nabla^2\theta^k(u) \cdot \G(u) \right] = \sum_{x \in E} \sum_{y \in E} \langle D^2 H(\S(u))(e_x, e_y), u^k \rangle \sum_{n \in [K]} \left( C(\S(u)) u^n \right)_{xy} \\ + \sum_{x \in E} \sum_{y \in E} \left( \langle D H(\S(u))(e_x), e_y \rangle + \langle D H(\S(u))(e_y), e_x \rangle \right) \left( C(\S(u)) u^k \right)_{xy}.
\end{multline*}
Using the linearity of $ Y \mapsto C(X) Y $ in the first line and \eqref{conservation_2} in the second line, we obtain
\begin{multline*}
	\mathrm{Tr}\left[\G(u)^* \cdot \nabla^2\theta^k(u) \cdot \G(u) \right] = \sum_{x \in E}\sum_{y \in E} \left\lbrace \l D^2H(\S(u))(e_x,e_y),u^k \r (C(\S(u))\S(u))_{xy} \right. \\ \left. - \l D^2H(\S(u))(e_x,e_y),\S(u) \r (C(\S(u))u^k)_{xy} \right\rbrace.
\end{multline*}
In particular, if $ u^k = \theta^k \tilde{h} $ and $ \sum_k \theta^k = 1 $, then
\begin{align*}
	\mathrm{Tr}\left[\G(u)^* \cdot \nabla^2\theta^k(u) \cdot \G(u) \right] & = \frac{1}{2} \sum_{x \in E}\sum_{y \in E} \left\lbrace \theta^k \l D^2H(\tilde{h})(e_x,e_y),\tilde{h} \r (C(\tilde{h})\tilde{h})_{xy} - \l D^2H(\tilde{h})(e_x,e_y),\tilde{h} \r \theta^k (C(\tilde{h})\tilde{h})_{xy} \right\rbrace \\
	                                                                       & = 0.
\end{align*}
This concludes the proof of the first point of the lemma.

Next, write
\begin{align*}
	\lf   \mathbf{G}(u) \cdot \nabla \theta^k(u), \mathbf{G}(u) \cdot \nabla \theta^{k'}(u) \rf
	 & ={\sum_{\underset{l\in[K]}{z\in E}}\sum_{\underset{m\in [K]}{w\in E}}
	\sum_{\underset{n\in [K]}{v\in E}}
	({\bf G} (U(s)))_{zl,wm}\nabla \theta^k(U)_{wm}
	({\bf G} (U(s)))_{zl,vn} }          \nabla \theta^{k'}(U)_{vn}           \\
	 & =	\sum_{w\in E}\sum_{v\in E}\sum_{n\in [K]}
	\nabla \theta^{k}(U)_{wn}\nabla \theta^{k'}(U)_{vn}
	(C(\S(U))U^n)_{wv}
\end{align*}
Combining \eqref{conservation_1} and Proposition~\ref{prop:scalar_form}, we obtain that
\begin{equation*}
	\langle D H(\tilde{h})(e_x), \tilde{h} \rangle = - \langle H(\tilde{h}), e_x \rangle = - h_x.
\end{equation*}
Putting this together with~\eqref{expr_nabla_thetak}
shows that if $ u^k = \theta^k \tilde{h} $ and $ \sum_k \theta^k = 1 $, then
\[
	\nabla \theta ^k(u)_{wn}=\l DH(\tilde h)(e_w), (\theta^k-\mathbf{1}_{k=n})\tilde h \r
	=-(\theta^k-\mathbf{1}_{k=n}) h_w
\]
In addition, we have $ C(\tilde{h}) \tilde{h} = a a^*(\tilde{h}) $.
This implies that, for $u\in \Gamma^K$,
\begin{align*}
	\lf   \mathbf{G}(u) \cdot \nabla \theta^k(u), \mathbf{G}(u) \cdot \nabla \theta^{k'}(u) \rf
	 & =\sum_{w\in E}\sum_{v\in E}\sum_{n\in [K]} (\theta^k(u)-\mathbf{1}_{k=n})
	(\theta^{k'}(u)-\mathbf{1}_{k=n}) (C(\tilde{h})\theta^n(u) \tilde{h})_{wv} h_w h_v                                              \\
	 & =\left(\sum_{w\in E}\sum_{v\in E} (aa^*(\tilde h))_{wv} h_w  h_v \right)\left(\sum_{n\in [K]} (\theta^k(u)-\mathbf{1}_{k=n})
	(\theta^{k'}(u)-\mathbf{1}_{k=n}) \theta^n(u)\right).
\end{align*}
The first bracket is no other than $ \Sigma^2 $ and, since $ \sum_n \theta^n = 1 $,
\begin{equation*}
	\sum_{n \in [K]} (\theta^k - \ind_{j=k})(\theta^{k'} - \ind_{j=k'}) \theta^n = \theta^k \theta^{k'} - 2 \theta^k \theta^{k'} + \ind_{k=k'} \theta^k=\ind_{k=k'} \theta^k-\theta^k\theta^{k'}.
\end{equation*}
This concludes the proof of the lemma.

\subsection{An alternative proof for $\infty$-decomposable SDEs}
\label{sec:2nd_proof}

\subsubsection{Structure of the proof}
\label{sec:structure_proof}

As in the first proof, we assume that the set $\B$ from Assumption~\ref{A2} is relatively compact.
The proof of Theorem~\ref{th:main_result} is divided into four main steps.
\begin{itemize}
	\item [(i)] First, following the strategy of Katzenberger~\cite{katzenberger1990solutions}, we apply Ito's formula to $(\theta^k)_{k\in [K]}$. This shows that, in a neighbourhood of the stable
	      manifold $\Gamma^K$, the generator of $(\theta^k)_{k\in [K]}$
	      resembles that of a $(K-1)$-dimensional Wright-Fisher diffusion.
	      The next steps (ii) - (v)  are devoted to proving that the convergence of the
	      trajectories follows from this observation.

	      Essentially, we prove the following stability result.
		      {\begin{prop}
				      \label{prop:stability_p}
				      Under the assumptions of Theorem~\ref{th:main_result},
				      for all $\vep>0$,
				      \[
						 \min_{u\in\Gamma^K} \| U^N_{tN/\Sigma^2} - u \|
						  \Rightarrow 0 \quad \text{in $D((\vep,T], \R)$}.
				      \]
			      \end{prop}
			      This proposition is an application of \cite[Theorem 5.1]{katzenberger1990solutions}. Here, we show that, in the $\infty$-decomposable case, this result follows from the stability of the Perron-Frobenius
			      projections. Unlike Katzenberger's method, this approach is
			      expected to extend readily to higher dimensions. The proof of
			      Proposition~\ref{prop:stability_p} is carried out in steps (ii)-(iv).
		      }
	      Theorem~\ref{th:main_result} then follows from a standard convergence result for continuous semimartingales (see Appendix~\ref{appendix:JacodShiryaev}), see Section~\ref{sec:proof_th}.
	\item[(ii)] {\bf Initial condition.} We  first explain why the
	      assumption on the initial condition $U_0^N$   in Theorem~\ref{th:main_result},
	      \begin{equation}\tag{$\mathcal{H} 1$}
		      \label{eq:H1}
		      \forall N\in \N, \quad U_0^N=U_0 \quad  \text{for some}\;
		      U_0\in \Bs,
	      \end{equation}
	      can be replaced by the following condition: for $N$ large enough,
	      \begin{equation}\tag{$\mathcal{H} 2$}
		      \label{eq:H2}
		      {
		      \| \S(U_0^N)-\tilde{h}\|+\max_{j\in [K]}\|U_0^{N,j}-\pi^j(U_0)\| < N^{-\alpha/2}
		      \quad \text{for some}\; \alpha\in(0,1).}
	      \end{equation}
	      Although this second assumption may seem more restrictive, we  show that, given
	      \eqref{eq:H1}, the solution of the SDE \eqref{eq:matrix_valued_SDE} will
	      satisfy \eqref{eq:H2} with high probability after a time of the order of $ \log(N) $. The argument is standard and we include it in  Section~\ref{sec:initial} for completeness.

	\item[(iii)] {\bf Stability {of Perron-Frobenius projections}.} Second, we show that the population profile remains
	      close to the equilibrium profile $ \tilde{h} $ (Section~\ref{sect:stabil_total})
	      and that
	      its composition matrix is well approximated by a set of
	      \textit{dynamical} Perron-Frobenius-like projections (Section~\ref{sec:proj_perron_frobenius}).
	      More precisely, we prove that if,  {for some} $\alpha\in(0,1)$, the initial condition $U_0^N$ satisfies, for $N$ large enough,
	      \begin{equation}\tag{$\mathcal{H} 3$}
		      \label{eq:H3}
		      {
		      \| \S(U_0^N)-\tilde{h}\|+\max_{j\in [K]}\|U_0^{N,j}-
		      P(U_0^{N,j})\| \leq N^{-\alpha/2},}
	      \end{equation}
	      then, for every $\beta<\alpha$,  with high probability and for $N$ large enough,
	      \begin{equation*}
		      {
		      \forall t\in[0,N],
		      \quad
		      \| \S(U_t^N)-\tilde{h}\|<3M_2N^{-\alpha/2}\quad \text{and} \quad
		      \max_{j\in [K]}\|U_t^{N,j}-
		      P(U_t^{N,j})\| < N^{-\beta/2}},
	      \end{equation*}
	      where $M_2$ is defined in \eqref{eq:J_dec} and $P$ is as in \eqref{eq:pf_proj}.

	\item[(iv)] {\bf Stability of Katzenberger projections.}
	      Finally, we show that the Perron-Frobenius projections constitute
	      a good proxy for the Katzenberger projections, i.e. that
	      \begin{equation*}
		      \pi^j(U^N_t) \approx  P (U^{N,j}_t)
	      \end{equation*}
	      to obtain Proposition~\ref{prop:stability_p} from (ii) and (iii).
	      The proof of this fact
	      relies heavily on Proposition~\ref{prop:scalar_form}.
\end{itemize}

\subsubsection{Ito's formula}
\label{sec:Ito}

\begin{prop}[Ito's formula]
	\label{lem:ito}
	{Let $\kappa^N=\inf\{t>0:U_t^N\notin \Bs\}$}.
	For all $ 0 \leq t \leq {\kappa^N}$,
	\begin{equation*}
		d \left( \theta^{k}(U^N_t) \right)
		= \frac{1}{N} B^k(U^{N}_t) dt+\frac{1}{\sqrt{N}}dY^{N,k}_t,
	\end{equation*}
	with
	\begin{equation*}
		B^k(u) = \frac{1}{2} \mathrm{Tr}\left[\G(u)^* \cdot \nabla^2 \theta^k(u) \cdot \G(u)\right],
	\end{equation*}
	and $ (Y^{N,k}_t, t \geq 0, k \in [K]) $ is a continuous martingale with quadratic variation process
	\begin{equation*}
		d[Y^{N,k}, Y^{N,k'}]_t = \C_{k,k'}(U^N_t) dt,
	\end{equation*}
	where $\C_{k,k'}$ is as in \eqref{eq:def_Cov}.
\end{prop}

\begin{proof}
	By \eqref{eq:matrix_valued_SDE} and Ito's formula, for all $ t < {\kappa^N} $,
	\begin{multline*}
		d\theta^k(U^N_t) = \lf \nabla \theta^k (U^N_t),  F(\S(U^N_t)) U^N_t dt \rf
		+ \frac{1}{\sqrt{N}} \lf\nabla \theta^k(U^N_t), \G(U^N_t) \cdot d \W_t\rf \\
		+\frac{1}{2 N} \mathrm{Tr}\left[\G(U^N_t)^* \cdot \nabla^2 \theta^k(U^N_t) \cdot \G(U^N_t)\right] dt.
	\end{multline*}
	Since $ U^N_t \in \Bs $ for $ t < \kappa^N $, the first term vanishes by \eqref{katzenberger_cancellation}.
	We thus have
	\[ B_t^k(U_t^N)=\frac{1}{2} \mathrm{Tr}\left[\G(U^N_t)^* \cdot \nabla^2 \theta^k(U^N_t) \cdot \G(U^N_t)\right] dt\]
	and
	\begin{equation*}
		d Y^{N,k}_t = \lf\nabla \theta^k(U^N_t), \G(U^N_t) \cdot d \W_t\rf.
	\end{equation*}
	It is then clear that
	\begin{align*}
		d [Y^{N,k}, Y^{N,k'}]_t & = \lf   \mathbf{G}(u) \cdot \nabla \theta^k(u), \mathbf{G}(u) \cdot \nabla \theta^{k'}(u) \rf\ dt= \C_{k,k'}(U_t^N)dt.
	\end{align*}

\end{proof}

{We now state a simple corollary of Lemma~\ref{lemma:katzenberger_projections},
which will be needed to establish the convergence of the processes $(\theta^k(U_t^N))$
to the Wright-Fisher diffusion via the convergence of their generators.}

\begin{lem} \label{lemma:approx_B_C}
	There exists a constant $ M > 0 $ such that, for all $ u \in \Bs $,
	\begin{equation} \label{bound_B}
		\max_{k \in [K]} | B^k(u) | \leq M \max_{k \in [K]} \| u^k - \pi^k(u) \|,
	\end{equation}
	and
	\begin{equation} \label{bound_C}
		\max_{k, k' \in [K]} | \C_{k,k'}(u) - \Sigma^2 \C_{k,k'}^{WF}(u) | \leq M \max_{k \in [K]} \| u^k - \pi^k(u) \|,
	\end{equation}
	where $\C_{k,k'}$ and $\C_{k,k'}^{WF}$ are defined in \eqref{eq:def_Cov}.
\end{lem}

\begin{proof}
	This is a direct consequence of Lemma~\ref{lemma:katzenberger_projections}. In particular,  since
	the maps $(\theta^k)$ are three times continuously differentiable on
	$\B$,  it follows that the maps $B^k$ and $\C_{k,k'}$ are Lipschitz on $\B$.
\end{proof}

\subsubsection{Initial condition}
\label{sec:initial}
In this section, we prove point (i).
The main idea consists in controlling the distance between $ V^N_t:=\S(U^N_t) $ and its deterministic approximation solving \eqref{eq: deterministic} up until the first time that $ \| V^N_t - \tilde{h} \| \leq N^{-\alpha/2} $, and thus show that this time is negligible compared to $ N $ with high probability as $ N \to \infty $.

Recall the definition of $\gamma_3$ and $M_3$ from \eqref{eq:HB1}. Without loss of
generality, one can assume that these constants are also such that
\begin{equation}
	\label{eq:HB2}
	\forall u_0\in\Bs, \; \forall t\in[0,\infty),	\quad \| \phi(t,\S(u_0))-\tilde h\|+
	\max_{j\in [K]}\| \varphi(t, \mathcal{S}(u_0), u_0^j) - \pi^j(u_0) \|
	\leq M_3 e^{- \gamma_3 t}.
\end{equation}

\begin{lem} \label{lemma:initial_phase}
	Under the assumptions of Theorem~\ref{th:main_result}, there exists $ \alpha_0 \in (0,1) $ such that, for all $ \alpha \in (0,\alpha_0) $, there exists $ c > 0 $ such that
	\begin{equation*}
		\lim_{N \to \infty} \P\left( \left(\max_{j \in [K]}
		\| U^{N,j}_{c \log(N)} - {\pi^j}{(U_0)}\|\right)
		{\vee \|V^N_{c \log(N)} - \tilde h\|\geq }
		N^{-\alpha/2} \right) = 0.
	\end{equation*}
\end{lem}

\begin{proof}
	For $ \varepsilon \in (0, {\gamma_3}) $, set $ c := - \frac{\alpha}{2(  \varepsilon-\gamma_3)} $ (which is positive) and $ T_N = c \log(N) $.
	Then, by \eqref{eq:HB2},
	\begin{equation*}
		\| \phi(T_N,\S(U_0))-\tilde h\|+
		\max_{1\leq j\leq K}\|\varphi(T_N, \mathcal{S}(U_0), U_0^j) -  {\pi^j}{(U_0)}
		\| \leq M N^{-\frac{\alpha}{2} \frac{\gamma}{\gamma- \varepsilon}} = o(N^{-\alpha/2}).
	\end{equation*}
	{Let $ \tilde \kappa^N $ denote the first time that $ V^N_t \notin \B $.}
	Let us abuse notation and denote, for $ u \in \M $,
	\begin{equation*}
		b(u) = F(\mathcal{S}(u)) u.
	\end{equation*}
	Let $ C_b > 0 $ be such that, for all $ u, u' \in \M $ such that $ \mathcal{S}(u), \mathcal{S}(u') \in \B $,
	\begin{equation} \label{Lipschitz_b}
		\| b(u) - b(u') \| + {\|b(\S(u))-b(\S(u'))\|}\leq C_b \| u - u' \|.
	\end{equation}
	We see from \eqref{eq:SDE_tot_pop}, \eqref{eq:SDE_fraction}, \eqref{eq:matrix_valued_SDE} and the definitions of  {the flows $ \varphi $ and $ \phi $} that
	\begin{align*}
		U^{N}_t - ( {\varphi(t,V_0,U_0^j)})_j & = \int_{0}^{t} \left( b(U^N_s) - b( {(\varphi(s,V_0,U_0^j))_j}) \right) ds
		+ \frac{1}{\sqrt{N}} \int_{0}^{t} \G(U^N_s) \cdot d\W_s,                                                           \\
		V^{N}_t -  {\phi(t,V_0)}              & = \int_{0}^{t} \left( b(U^N_s) - b( {\phi(s,V_0)}) \right) ds
		+ \frac{1}{\sqrt{N}} \int_{0}^{t} a(V^N_s)dW_s .
	\end{align*}
	Using \eqref{Lipschitz_b} and Grönwall's inequality, we obtain
	\begin{align*}
		\sup_{t \in [0, T_N  \wedge \tilde \kappa^N]} \| U^N_{t} -  ( {\varphi(t,V_0,U_0^j)})_j \|
		 & \leq \frac{1}{\sqrt{N}} \left( \sup_{t \in [0, T_N  \wedge \tilde \kappa^N]} \left\|
		\int_{0}^{t} \G(U^N_s) \cdot d\W_s \right\| \right) e^{C_b T_N},                        \\
		\sup_{t \in [0, T_N  \wedge \tilde \kappa^N]} \| V^N_{t} -  ( {\phi(t,V_0)}) \|
		 & \leq \frac{1}{\sqrt{N}} \left( \sup_{t \in [0, T_N  \wedge \tilde \kappa^N]} \left\|
		\int_{0}^{t} a(V^N_s)  dW_s \right\| \right) e^{C_b T_N}.
	\end{align*}
	For $ \delta > 0 $ to be chosen later, let $ A_N $ and $B_N$ denote the events
	\begin{equation*}
		A_N := \left\lbrace \sup_{t \in [0, T_N  \wedge \tilde \kappa^N]} \left\| \int_{0}^{t} \G(U^N_s) \cdot d\W_s \right\| \leq N^\delta \right\rbrace
		\quad \text{and} \quad
		B_N := \left\lbrace \sup_{t \in [0, T_N  \wedge \tilde \kappa^N]} \left\| \int_{0}^{t} a(V^N_s)  dW_s \right\| \leq N^\delta \right\rbrace.
	\end{equation*}
	By Doob's inequality, using the fact that $ u \mapsto \G(u) \cdot \G(u)^T $ is bounded on $ \B $, there exists $ C > 0 $ such that, for all $ N \geq 1 $,
	\begin{equation*}
		\P\left( A_N^c \right) \leq C \frac{T_N}{N^{2\delta}},
	\end{equation*}
	which tends to zero as $ N \to \infty $ by the definition of $ T_N $.
	Without loss of generality, we can assume that $C$ can be chosen such that
	\begin{equation*}
		\P\left( B_N^c \right) \leq C \frac{T_N}{N^{2\delta}}.
	\end{equation*}
	(This can be proven using the exact same arguments).
	It follows that, on the event $ {A_N \cap B_N} $,
	\begin{equation*}
		\sup_{t \in [0, T_N  \wedge \tilde \kappa^N]}
		\left(
		\| U^N_{t} - ({\varphi(t,V_0,U_0^j)})_j \|
		+\|V_t^N-{\phi(t,V_0)}\|
		\right) \leq 2 N^{\delta + c C_b - 1/2}.
	\end{equation*}
	By  definition of $ c $, the exponent on the right-hand side of the above equation is
	\begin{equation*}
		\delta - \frac{\alpha C_b}{2(\varepsilon-\gamma_3)} - 1/2.
	\end{equation*}
	Then, as long as $ \alpha < \frac{1}{1 + \frac{C_b}{\gamma_3}} $, we can find $ \varepsilon > 0 $ and $ \delta > 0 $ sufficiently small such that this exponent is strictly less than $ -\alpha /2 $.
	Furthermore, for such a choice of $ \alpha $, $ \delta $ and $ \varepsilon $, on the event $ A_N\cap B_N $, we must have $ T_N \leq \tilde \kappa^N $ for all $ N $ large enough since $ \phi(t,V_0) \in \B $ for all $ t \geq 0 $.
	This concludes the proof.
\end{proof}
\subsubsection{The total population size}
\label{sect:stabil_total}
Next, we show that, under \eqref{eq:H3}, the probability to observe fluctuations of order
$N^{-\alpha/2}$ in the total
population size $V_t^N=\S(U^N_t)$ on the time interval $[0,N]$ is polynomially small.
(Note that these fluctuations are much larger than those expected by
the central limit theorem). The proof of this result relies on a coupling argument.

Consider the process $(\tilde V_t^N, t \geq 0)$ defined as the unique weak solution
to the SDE
\begin{equation}\label{eq:vt}
	d\tilde V_t^N = b(\tilde V_t^N) dt
	+\frac{1}{\sqrt{N}}\tilde  a(\tilde V_t^N)d W_t, \quad \text{with} \quad \tilde V_0^N=V_0^N,
\end{equation}
where $b$ and $W$ are defined as in \eqref{eq:SDE_tot_pop}, and $\tilde a$ is defined
\begin{eqnarray*}
	\tilde a: & \mathbb{R}^{E} & \to  \mathbb{R}^{E\times E}, \\
	& V &  \mapsto   a(p(V)),
\end{eqnarray*}
and $p:\mathbb{R}^{E}\to \mathbb{R}^{E}$ refers to the projection on the closed ball
$\bar B(\tilde h, \frac{1}{2}\min \tilde{h}_x)$. It is clear that
$\tilde V_t^N$ and $V_t^N$ can be coupled so as to coincide until the first time
$\tilde V_t^N$ exits the ball $\bar B(\tilde h, \frac{1}{2}\min \tilde{h}_x)$.
On the other hand, Taylor's theorem yields that
\begin{equation*}
	d(\tilde{V}^N_t-\tilde{h})=d\tilde{V}^N_t=J(\tilde{V}^N_t-\tilde h)dt
	+R_2(\tilde{V}^N_t-\tilde h)dt+\frac{1}{\sqrt{N}}\tilde a(\tilde V_t^N)d\tilde W_t,
\end{equation*}
where $J$ is as in Assumption~\ref{A2} and $R_2$ denotes the second-order Taylor
remainder. Furthermore, we know from Assumption~\ref{A:regularity} that
there exists a constant $M>0$ such that

\begin{equation}
	\label{eq:t_remainder}
	\forall v\in {\bar B\left(\tilde h, \frac{1}{2}\min \tilde{h}_x\right)},  \quad
	\|R_2(v-\tilde h)\|\leq M\|v-\tilde h\|^2.
\end{equation}

By Duhamel's principle,
\begin{equation*}
	\tilde V_t^N - \tilde h = \underbrace{e^{tJ}(\tilde V_0^N - \tilde h)}_{\textstyle :=c_1(t)} +
	\underbrace{\int_0^t e^{(t-s)J} R_2(\tilde V_s^N - \tilde h) ds}_{\textstyle :=c_2(t)} +
	\underbrace{\frac{1}{\sqrt{N}} e^{tJ} X_t^N}_{\textstyle :=c_3(t)},
\end{equation*}
where
$
	X^N_t:=\int_0^te^{-sJ} \tilde a(\tilde V^N_s)d\tilde{W}_s
$
is a martingale.
To bound the fluctuations of $V^N$, we will first control the fluctuations of $\tilde V^N$. The result for $V^N$ will then follow from a straightforward coupling argument.
In turn, controlling the fluctuations of $\tilde V^N$ boils down to
controlling $c_1,c_2$ and $c_3$.

\begin{lem}\label{lem:total}
	Assume that \eqref{eq:H3} holds. Let $M_2$ be as in \eqref{eq:J_dec}.
	Define
	$${{\tilde \tau}_{pop}^N}:=\inf\{t>0:\tilde  V_t^N\notin B(\tilde{h},3M_2N^{-\alpha/2})\}.$$
	For any $\ell \in \N$ and $T>0$, there exist  $M_{\ref{lem:total}}>0$ and
	$N_{\ref{lem:total}}>0$
	such that

	\begin{equation*}
		\forall N\geq N_{\ref{lem:total}},\qquad
		\mathbb{P}\left({{\tilde \tau}_{pop}^N}\leq T\right)\leq M_{\ref{lem:total}} N^{-{\ell(1-\alpha)/2}}.
	\end{equation*}
	Under the additional assumption that $T>\frac{1}{\gamma_2}\log(M_2),$
	$M_{\ref{lem:total}}$ and $N_{\ref{lem:total}}$ can be chosen such that 
		\[
			\forall N\geq N_{\ref{lem:total}},\qquad
			\mathbb{P}\left( \tilde V^N_T\notin  B(\tilde h,N^{-\alpha/2}),
		\ T<{{\tilde \tau}_{pop}^N} \right)\leq M_{\ref{lem:total}} N^{-{\ell(1-\alpha)/2}}.\]

\end{lem}
\begin{proof} We first bound $c_1$ and $c_2$ on the event $\{t\leq {{\tilde \tau}_{pop}^N}\}$.
	We see from \eqref{eq:J_dec} that, for all $t\in[0,T]$,
	\begin{equation}
		\label{eq:bound_b1}
		\|c_1(t)\|\leq { M_2e^{-\gamma_2 t}} N^{-\alpha/2}
		\leq M_2N^{-\alpha/2},
	\end{equation}
	and from \eqref{eq:J_dec} and \eqref{eq:t_remainder} that,
	for all $t\leq {\tilde \tau^N}$,
	\begin{equation}
		\label{eq:bound_c2}
		\left\|c_2(t)\right\|\leq \int_0^t\|e^{(t-s)J}\| \,
		\|R_2(\tilde V_s^N - \tilde h)\|ds
		\leq M N^{-\alpha},
	\end{equation}
	for some $M>0$.
	Fix  $N_0>0$ such that, for all $N>N_0$, we have $M N^{-\alpha}<  M_2N^{-\alpha/2}$.
	We thus obtain that
	\begin{equation*}
		\forall N >N_0, \quad \P({{\tilde \tau}_{pop}^N}\leq T)\leq
		\P\left(\sup_{0\leq t\leq T}\|c_3(t)\|\geq M_2N^{-\alpha/2}\right)
		\leq \P\left(N^{-1/2}\sup_{0\leq t\leq T}\| X_t^N\|\geq
		N^{-\alpha/2}\right),
	\end{equation*}
	where we used \eqref{eq:J_dec} to obtain the last inequality.
	It then follows from Markov's inequality that, for any $ \ell \in \N $,
	\begin{equation*}
		\P\left( \sup_{0\leq t\leq T}\|X_t^N\|\geq
		N^{(1-\alpha)/2}\right)\leq
		N^{\ell(\alpha-1)/2}\Ee\left[
			\sup_{0\leq t \leq T}\|X_T^N\|^\ell\right].
	\end{equation*}
	Since $\tilde a$ is bounded, the Burkholder-Davis-Gundy
	inequality  (see e.g.~\cite{karatzas1991brownian}) entails
	\begin{equation}
		\label{eq:BDG}
		\mathbb{E}\left[\sup_{0\leq t \leq T}\|X^N_T\|^\ell\right]\leq M,
	\end{equation}
	for some positive constant $M$ that depends on $\ell$ and $T$ but that does not depend on $N$.
	This yields the first part of the lemma.

	Let us now consider $\eta>0$ and $t_\eta>0$ such that
	${M_2e^{-\gamma_2 t}}< (1-2\eta)$ for all $t\geq t_\eta.$
	Since $T>\frac{1}{\gamma_2}\log(M_2)$,  $\eta$ can be chosen small enough so
	that $t_\eta<T$. From~\eqref{eq:bound_b1} we get
	\begin{equation*}
		\|c_1(T)\|\leq (1-2\eta)N^{-\alpha/2}.
	\end{equation*}
	It then follows from \eqref{eq:bound_c2} that one can choose
	$N_0$ large enough that
	\begin{equation*}
		\forall N>N_0, \quad
		\|c_1(T)\|+\|c_2(T)\|< (1-\eta)N^{-\alpha/2}.
	\end{equation*}
	Thus,  Markov's inequality yields that, for all $ \ell \in \N $,
	\begin{align*}
		\forall N>N_0, \quad
		\mathbb{P}\left(\tilde V^N_T \notin
		\bar B(\tilde h,N^{-\alpha/2}), \ T<{{\tilde \tau}_{pop}^N} \right)
		 & \leqslant \P\left( \|X_T^N\|\geq \eta N^{(1-\alpha)/2}\right) \\
		 & \leqslant \eta^{-1}N^{\ell(\alpha-1)/2}
		\mathbb{E}\left[\|X_T^N\|^\ell\right].
	\end{align*}
	This combined with \eqref{eq:BDG} concludes the proof.
\end{proof}

This translates to the following lemma for $ (V^N_t, t \geq 0) $.

\begin{lem}\label{cor:total}
	Assume that \eqref{eq:H3} holds. Let $M_2$ be as in \eqref{eq:J_dec}.
	
	For $\ell\in\mathbb{N}$ and $T>0$,
	let  $M_{\ref{lem:total}}$ and $N_{\ref{lem:total}}$
	be as in Lemma~\ref{lem:total}.
	Define
	$${\tau}^N_{pop}:=\inf\{t>0:  V_t^N\notin B(\tilde{h},3M_2N^{-\alpha/2})\}.$$
	Then
	\begin{equation*}
		\forall N\geq N_{\ref{lem:total}},\qquad
		\mathbb{P}\left({{\tilde \tau}_{pop}^N}\leq T\right)\leq M_{\ref{lem:total}} N^{-{\ell(1-\alpha)/2}}.
	\end{equation*}
	Under the additional assumption that $T>\frac{1}{\gamma_2}\log(M_2),$
	we have 
		\[
			\forall N\geq N_{\ref{lem:total}},\qquad
			\mathbb{P}\left( \tilde V^N_T\notin  B(\tilde h,N^{-\alpha/2}),
		\ T<{{\tilde \tau}_{pop}^N} \right)\leq M_{\ref{lem:total}} N^{-{\ell(1-\alpha)/2}}.\]

\end{lem}
\begin{proof}
	Let $\tilde{\tau}^N$ be as in Lemma~\ref{lem:total}.
	The result follows from a straightforward coupling argument:
	one can couple $V^N$ and $\tilde V^N$ so that they coincide until $\tilde \tau^N$,
	which implies that
	\begin{equation*}
		\P(\tau^N_{pop} =\tilde \tau^N_{pop})=1.
	\end{equation*}
	Lemma \ref{lem:total} then entails the result.
\end{proof}

\begin{prop}\label{prop:Teps2}
	Assume that \eqref{eq:H3} holds.
	Let $\tau^N$ be as in Lemma~\ref{cor:total}.
	For any $\ell \in \N$ and $T>0$, there exist $ M_{\ref{prop:Teps2}}>0$ and
	$N_{\ref{prop:Teps2}}>0$ such that
	\begin{equation*}
		\forall N>N_{\ref{prop:Teps2}}, \quad
		\mathbb{P}\left(\tau^N_{pop}\leq {T}N\right)\leqslant M_{\ref{prop:Teps2}}
		N^{{1-\ell(1-\alpha)/2}}.
	\end{equation*}
\end{prop}

{\begin{rem}
	Note that  $\ell$ can be chosen arbitrarily in the above lemmas.
	This is because our polynomial bounds could, in fact, be replaced by exponential bounds.
	For simplicity, we do not prove this, as it is not required for our purposes.
\end{rem}
}
\begin{proof}
	The proof of the result relies on the Markov property and Lemma \ref{cor:total}.

	Let us first assume that $T>\frac{1}{\gamma_2}\log(M_2)$. Let $M_{\ref{lem:total}}$ and
	$N_{\ref{lem:total}}$ be as
	in Lemma~\ref{lem:total}. For $j\in\{1,..., N \}$, set $t_j=jT$ and consider
	the good event
	\begin{equation*}
		\mathcal G^N_j
		:=\{\forall i \in \{1,...,j\}, V^N_{t_i}\in B(\tilde h,N^{-\alpha/2}), \
		\text{and} \ V^N_{s}\in B(\tilde h,3M_2N^{-\alpha/2}), \forall s\in[t_{i-1},t_{i}] \}.
	\end{equation*}
	Lemma~\ref{cor:total} along with a union bound show that, for all $N>N_{\ref{lem:total}}$,
	\begin{align*}
		\P(\mathcal (G_1^N)^c) & = \P(\{V_T^N\in B(\tilde h,N^{-\alpha/2}),\tau^N_{pop}>T \}^c)
		=\P(\tau^N_{pop}\leq T \ \text{or} \ V^N_T\notin B(\tilde h,N^{-\alpha/2}))                                                  \\
					& \leq \P(\tau^N_{pop}\leq T)+\P(V_T^N\notin B(\tilde h,N^{-\alpha/2}),\tau^N_{pop} > T)        \\
					& \leq 2M_{\ref{lem:total}}N^{\ell(\alpha-1)/2},
	\end{align*}
	where the last line follows from Lemma~\ref{cor:total}.
	We remark that $ \mathcal G^N_{j+1}\subset \mathcal G^N_j$ so that
	$ (\mathcal G^N_j)^c\subset( \mathcal G^N_{j+1})^c$.
	Hence,
	\begin{equation*}
		\forall N>N_{\ref{lem:total}}, \quad
		\P((\mathcal G^N_{j+1})^c)=\P((\mathcal G^N_{j+1})^c\cap \mathcal G^N_j)
		+\P((\mathcal G^N_{j+1})^c\cap (\mathcal G_j^N)^c)
		=\P((\mathcal G_{j+1}^N)^c\cap \mathcal G^N_j)+\P((\mathcal G^N_j)^c).
	\end{equation*}
	The first term can be bounded as follows,
	\begin{align*}
		\forall N>N_{\ref{lem:total}}, \quad
		\P((\mathcal G^N_{j+1})^c\cap \mathcal G_j) & \leq \p{(\mathcal G^N_{j+1})^c\,|\,\mathcal G_j}
		=	\p{(\mathcal G^N_1)^c} \leq 2M_{\ref{lem:total}}N^{\ell(\alpha-1)/2}.
	\end{align*}
	Here, in the equality, we used the Markov property and the time-homogeneity of the process. It then follows by induction that
	\begin{equation*}
		\forall N>N_{\ref{lem:total}},
		\quad \forall j, \quad
		\P((\mathcal G^N_j)^c)\leq 2j M_{\ref{lem:total}}N^{\ell(\alpha-1)/2},
	\end{equation*}
	so that
	\begin{equation*}
		\forall N>N_{\ref{lem:total}}, \quad
		\P \left(\tau^N_{pop}\leq TN\right)\leq
		\P((\mathcal G^N_{N+1})^c)
		\leq MN^{\ell(\alpha-1)/2+1},
	\end{equation*}
	for some $M>0$.

	For $T \leq \frac{1}{\gamma_2}\log(M_2)$, the proof is essentially the same. We divide the interval into
	subintervals of length $\frac{2}{\gamma_2}\log(M_2)$, over which the argument applies. On the last interval,
	it suffices to show that the population configuration stays in the ball
	$B(\tilde{h}, 3 M_2 N^{-\alpha/2})$. This follows from the strong Markov
	property, the time-homogeneity of the process, and the fact that the first
	bound in Lemma~\ref{lem:total} does not require $T > \frac{1}{\gamma_2}\log(M_2)$.

\end{proof}

\subsubsection{Perron-Frobenius projections}
\label{sec:proj_perron_frobenius}
We now compare the size of each fraction to its dynamical Perron-Frobenius
projection. More precisely, we provide bounds on the quantities
\begin{equation}\label{eq: Ztkdef}
	Z_t^{N,k} := U_t^{N,k} - \l U_t^{N,k},h  \r\tilde h,
	\quad k\in [K].
\end{equation}
As a first step, we bound the norm of $Z_t^{N,k}$ on a time
interval of length of order $1$.

\begin{lem}
	\label{lem:subf}
	Assume that \eqref{eq:H3} holds.
	Let $\beta\in (0,\alpha)$ and define
	$$\tau^{N,k}_{PF}:=\inf\left\{t>0:\|Z_t^{N,k}\|\geq N^{-\beta/2}\right\}.$$
	Then, for any $ \ell \in\mathbb{N}$ and $T>0$, there exist $M_{\ref{lem:subf}}>0$ and $N_{\ref{lem:subf}}>0$
	such that, for all $N>N_{\ref{lem:subf}}$ and $ k\in[K]$,
	\begin{equation*}
		\mathbb{P}\left(\tau^{N,k}_{PF} \leq T\right)\leq M_{\ref{lem:subf}} N^{-\ell(1-\alpha)/2},
		\quad \text{and} \quad
		\mathbb{P}\left(\|Z_T^{N,k}\|\geq N^{-\alpha/2},
		\ T \leq \tau^{N,k}_{PF} \right)\leq M_{\ref{lem:subf}} N^{-\ell(1-\alpha)/2}.
	\end{equation*}

\end{lem}

\begin{proof} Fix $k\in[K]$.
	First of all, we note that we only need to focus on what happens on the event
	$\{T<\tau^N\}$. Indeed, we know from Lemma~\ref{cor:total} that
	\begin{equation}
		\label{eq: tauktau}
		\mathbb{P}\left(\tau^{N,k}_{PF}\leq T\right) \leq
		\mathbb{P}\left( \tau^{N,k}_{PF}\leq T,\, T \leq \tau^N \right)
		+ M_{\ref{lem:total}}N^{-\ell(1-\alpha)/2},
	\end{equation}
	and that
	\begin{equation}\label{eq: ZTktau}
		\mathbb{P}\left( \|Z_T^{N,k}\|\geq N^{-\beta/2}, \ T \leq \tau^{N,k}_{PF} \right)
		\leq \mathbb{P}\left(\|Z_T^{N,k}\|\geq N^{-\beta/2}, \, T \leq \tau^{N,k}_{PF},\, T \leq \tau^N \right)
		+ M_{\ref{lem:total}}N^{-\ell(1-\alpha)/2}.
	\end{equation}
	From now on we will restrict ourselves to the event $\{T<\tau\}$.

	Define
	\[
		A := F(\tilde h){\big|}_{\text{vect}(h)^\perp}.
	\]
	Observe that if $z\in \text{vect}(h)^\perp$, then $F(\tilde h)z\in
		\text{vect}(h)^\perp$; indeed, we have
	\begin{equation}
		\label{eq:ort_sp}
		\l  F(\tilde h)z,h   \r
		= h^* F(\tilde h)z = 0
	\end{equation} by Assumption~\ref{A3}. Therefore,  the maps
	$A:\text{vect}(h)^\perp \to \text{vect}(h)^\perp$ and
	$e^{ A}: \text{vect}(h)^\perp \to \text{vect}(h)^\perp$ are well-defined.
	{It also follows from \eqref{eq:PR_dec}} that
	\begin{equation*}
		\forall t\in[0,\infty),  \quad \|e^{tA}\|\leq M_1e^{-\gamma_1 t}.
	\end{equation*}

	The proof of Lemma~\ref{lem:subf} is similar to that of Lemma~\ref{lem:total}.
	First, we note that $Z^{N,k}$ is solution to
	\begin{align*}
		dZ_t^{N,k} & = dU_t^{N,k}- \l dU_t^{N,k},h  \r\tilde h
		\\ & =
		F(\tilde h)U_t^{N,k}dt + (F(\S(U^{N}_t))-F(\tilde h))U_t^{N,k}dt
		-\l F(\S(U^{N}_t)-F(\tilde h))U_t^{N,k}dt, h  \r\tilde h
		\\ & \quad
		+ \frac{1}{\sqrt{N}}\sigma(\S(U_t^{N}),U_t^{N,k}) dW_t^k
		- \frac{1}{\sqrt{N}}\l \sigma(\S(U_t^{N}),U_t^{N,k}) dW_t^k,h  \r \tilde h\numberthis,
		\label{eq: dZt}
	\end{align*}
	where we used \eqref{eq:ort_sp} to see that $\l F(\tilde h)U_t^{N,k},h\r=0$.
	On the event $\{T<\tau^N\}$,  all the coefficients in \eqref{eq: dZt}
	are bounded  by a constant only depending on $\tilde h$ and $h$.
	Noting that $ F(\tilde{h})U^{N,k}_t = A Z^{N,k}_t $, it then follows from Duhamel's principle that
	\eqref{eq: dZt} can be written as
	\begin{equation}
		\label{eq: Ztk}
		Z_t^k = e^{t A} Z_0^k + \int_{0}^{t} e^{(t-s) A}
		\mu_s^k ds + \frac{1}{\sqrt{N}} e^{t A} X_t^k,
	\end{equation}
	where
	\begin{equation*}
		\mu^k_t:= (F(\S(U^{N}_t)-F(\tilde h))U_t^{N,k}
		+\l F(\S(U^{N}_t)-F(\tilde h))U_t^{N,k}, h  \r\tilde h,
	\end{equation*}
	and
	\begin{equation*}
		X^k_t:=\frac{1}{\sqrt{N}}
		\left(
		\int_{0}^{t} e^{-s A} \sigma(\S(U_s^{N}),U_s^{N,k}) dW_s^k
		- \int_{0}^{t} e^{-s A} \l \sigma(\S(U_s^{N}),U_s^{N,k}) dW_s^k,h  \r \tilde h
		\right)
	\end{equation*}
	is a martingale. The proof then goes along the exact same lines
	as the proof of Lemma~\ref{lem:total}.
	First, we see that, under \eqref{eq:H3}, for $N$ large enough,
	\begin{equation*}
		\forall t\in[0,\infty), \quad
		\|e^{tA}Z_0^k\|\leq M_1e^{-\gamma_1 t}N^{-\alpha/2}\leq \frac{1}{3}N^{-\beta/2}.
	\end{equation*}
	We then use that, on $\{T<\tau^N\}$
	\begin{equation*}
		\|F(\S(U^{N}_t))-F(\tilde h)\|\leq M N^{-\alpha/2}\quad \text{and}
		\quad \|U_t^{N,k}\| \leq  \|\S(U_t^N)\| \leq \|\tilde h\| + N^{-\alpha/2},
	\end{equation*}
	so that
	\begin{equation}\label{eq: Ztk2}
		\forall 0<t\leq T, \quad
		\left\|\int_{0}^{t} e^{(t-s) A} \mu_s^k ds \right\|
		\leq M\int_{0}^{t} e^{-\gamma_1(t-s) }
		\|\mu_s^k\| ds  \leq MN^{-\alpha/2} < \frac{1}{3}N^{-\beta/2},
	\end{equation}
	for  some $M>0$, and $N$ large enough. The martingale part is then bounded
	by applying  Markov's inequality
	and Burkholder-Davis-Gundy inequality (here, we recall that all the coefficients
	appearing in the martingale part of \eqref{eq: dZt} are bounded on $\{T<\tau^N\}$).
	This shows that for $N$ large enough
	\[
		\p{\tau^{N,k}_{PF} \leq T, \tau^N>T} \leq  MN^{-\ell(1-\alpha)}
	\]
	for some constant $M>0$, which together with \eqref{eq: tauktau} yields the
	first part of the lemma.
	The proof of the second part of the lemma relies on the same arguments as those used in the proof of Lemma~\ref{lem:total}
	and on \eqref{eq: ZTktau}.
\end{proof}

\begin{prop}\label{prop:taukTeps}
	Assume that \eqref{eq:H3} holds.
	Let $\beta$ and $(\tau^{N,k}_{PF})_{k=1}^K$ be as in Lemma~\ref{lem:subf}.
	Then for any $\ell\in \N$ and $T>0$, there exist $M_{\ref{prop:taukTeps}}>0$
	and $N_{\ref{prop:taukTeps}}>0$ such that,
	\begin{equation*}
		\forall N>N_{\ref{prop:taukTeps}}, \quad
		\mathbb{P}\left(\min_{k\in[K]}\tau^{N,k}_{PF}\leq TN\right)\leqslant
		M_{\ref{prop:taukTeps}}N^{-\ell(1-\alpha)/2+1}.
	\end{equation*}
\end{prop}

\begin{proof}
	The proof is essentially the same as that of
	Proposition~\ref{prop:Teps2}.

	Assume that $M_{\ref{lem:subf}}$
	and $N_{\ref{lem:subf}}$
	are as in Lemma~\ref{lem:subf}. For $j\in\{1,...,N \}$,
	set $t_j=jT$ and for all $k\in[K]$ consider the event
	\begin{equation*}
		\mathcal G^{N,k}_{j}:=\{\forall i \in \{1,...,j\},
		Z^{N,k}_{t_i}\in B(0,N^{-\alpha/2}), \ \text{and} \
		Z^{N,k}_{s}\in B(0,N^{-\beta/2}), \forall s\in[t_{i-1},t_{i}] \}.
	\end{equation*}
	Fix $k\in[K]$. Now exactly the same calculation as in the proof of Proposition~\ref{prop:Teps2} gives
	\begin{equation*}
		\forall N> N_{\ref{lem:subf}}, \quad \forall j\in\mathbb{N}, \quad
		\P((\mathcal G^{N,k}_j)^c)\leq M_{\ref{lem:subf}}
		N^{-\ell(1-\alpha)/2} .
	\end{equation*}
	The remainder of the proof is akin to that of Proposition~\ref{prop:Teps2}.
\end{proof}

\subsubsection{Proof of Proposition~\ref{prop:stability_p}}
\begin{lem}\label{lem:IC_Katzenberger}
	Assume that \eqref{eq:H2} holds. {Then \eqref{eq:H3} is satisfied for every $\alpha'<\alpha$} .
\end{lem}

\begin{proof}
	By the triangle inequality,
	\begin{equation*}
		\forall u\in \Bs
		\quad
		\|u^k-\l u^k,h\r \tilde h\|\leq \|u^k-\pi^k(u)\|+
		\|\l u^k,h\r \tilde h-\pi^k(u)\|.
	\end{equation*}
	It then follows from Proposition~\ref{prop:scalar_form}
	and the Cauchy-Schwarz inequality that
	\begin{equation*}
		\forall u\in \Bs
		\quad
		\|u^k-\l u^k,h\r \tilde h\|\leq \|u^k-\pi^k(u)\|+
		\|h\|\|\tilde h\| \|H(\S(u))-h\|.
	\end{equation*}
	Let $M$ be as in Corollary~\ref{cor:H_smooth}. Note that for $N$ large enough,
	$B(\tilde{h},N^{-\alpha/2})\subset \B$. Hence, under \eqref{eq:H2},
	$\|H(\S(U_0^N)))-h\|\leq M\|\S(U_0^N)-\tilde{h}\|$ for $N$ large enough.
		{Thus, under \eqref{eq:H2}, Assumption \eqref{eq:H3} holds for every $\alpha'<\alpha$.}
\end{proof}

\begin{lem}
	\label{lem:PR_Katz}
	Assume that \eqref{eq:H2} holds.
	Let {$\beta<\alpha'<\alpha$} and define
	\begin{equation}
		\label{eq:tau_hat}
		\tau^{N,k}_{Katz} :=
		\inf\{t>0:\: \|U_t^{N,k}-\pi^k(U^N_t)\|\geq 2N^{-\beta/2} \}.
	\end{equation}
	Then, for any $\ell\in\mathbb{N}$, and $T>0$, there exist $M_{\ref{lem:PR_Katz}}>0$
	and $N_{\ref{lem:PR_Katz}}>0$ such that
	\begin{equation*}
		\forall N>N_{\ref{lem:PR_Katz}},\quad
		\p{ \min_{k\in[K]}\tau^{N,k}_{Katz} < TN} <
		M_{\ref{lem:PR_Katz}}N^{1-\ell(1-\alpha')/2} .
	\end{equation*}
\end{lem}

\begin{proof}
	First, one can check (as in the proof of Lemma~\ref{lem:IC_Katzenberger}) that
		{\begin{equation}
				\label{eq:CS}
				\forall u\in \Bs
				\quad
				\|u^k-\pi^k(u)\|\leq \|u^k-\l u^k,h\r\tilde h \| +
				\|h\|\|\tilde h\| \|H(\S(u))-h\|.
			\end{equation}}
	Second, we know from Lemma~\ref{lem:IC_Katzenberger} that \eqref{eq:H3} holds {for some $\alpha'\in(\beta,\alpha)$}.
	Moreover, we know from Proposition~\ref{prop:Teps2} and Proposition~\ref{prop:taukTeps}
	that under \eqref{eq:H3},  for $N\geq N_{\ref{prop:Teps2}}\vee
		N_{\ref{prop:taukTeps}}$,
	the event
	\begin{equation*}
		E_{PF}^N:=\{
		\forall t\in[0,TN],
		\quad
		\| \S(U_t^N)-\tilde{h}\|<3N^{-\alpha'/2}\quad \text{and} \quad
		\max_{j\in [K]}\|U_t^{N,j}-
		\l U_t^{N,j},h \r \tilde h\| < N^{-\beta/2}
		\}
	\end{equation*}
	occurs with probability larger than
	$1-(M_{\ref{prop:Teps2}}+M_{\ref{prop:taukTeps}})N^{1-\ell(1-{\alpha'})/2}$. Let
	$M$ be as in Corollary~\ref{cor:H_smooth} and let
	$N_{\ref{lem:PR_Katz}}\geq  N_{\ref{prop:Teps2}}\vee
		N_{\ref{prop:taukTeps}}$ be large enough that $B(\tilde{h},N^{-\beta/2})\subset \B$ for all $N\geq N_{\ref{lem:PR_Katz}}$.
	It then follows from Corollary~\ref{cor:H_smooth} and \eqref{eq:CS} that, for $N\geq N_{\ref{lem:PR_Katz}}$, on the event $E_{PF}^N$,
	\begin{equation*}
		\forall t\in[0,TN],\quad \max_{1\leq j\leq K}
		\|U_t^{N,k}-\pi^k(U^N_t)\|\leq N^{-\beta/2}+ 3M\|h\|\|\tilde h\| N^{-\alpha'/2}.
	\end{equation*}
	It remains to choose $N$ large enough that the right-hand side of the above inequality
	is smaller than $2N^{-\beta/2}$ to conclude the proof.
\end{proof}

\begin{proof}[Proof of Proposition~\ref{prop:stability_p}]
	We know from Lemma~\ref{lemma:initial_phase} that $ U^{N}_{c \log(N)} $ satisfies \eqref{eq:H2} with high probability as $ N \to \infty $, and, by Lemma~\ref{lem:IC_Katzenberger}, that it also also satisfies \eqref{eq:H3} with high probability.
	Taking $ \ell > \frac{2}{1-\alpha} $ and using Proposition~\ref{prop:Teps2} and Lemma~\ref{lem:PR_Katz}, we see that the event
	\begin{equation}
		\label{eq:E_Katz}
		E^N_{Katz}:=\{\min(\tau_{pop}^N, \tau^{N,1}_{Katz},..., \tau^{N,K}_{Katz})
		> TN \}
	\end{equation}
	has a probability which tends to one as $ N \to \infty $.
	This concludes the proof.

\end{proof}

\subsubsection{Proof of Theorem~\ref{th:main_result}}
\label{sec:proof_th}

Our strategy to prove Theorem~\ref{th:main_result} is to apply Theorem~\ref{thm:JS} below to the $ \R^K $-valued process $ (X^N_t, t \geq 0) $ defined as
\begin{equation*}
	X^{N,k}_t := \theta^k\left(U^N_{\frac{Nt}{\Sigma^2}}\right).
\end{equation*}
Then, by Lemma~\ref{lemma:initial_phase}, we know that, for some $ c > 0 $,
\begin{equation*}
	X^{N,k}_{c \Sigma^2 \frac{\log(N)}{N}} \to \theta^k_0, \quad \text{ in probability as } \quad N \to \infty.
\end{equation*}
(Note that we do not necessarily have $ X^{N,k}_0 \to \theta^k_0 $, which is why the convergence in Theorem~\ref{th:main_result} is stated on $ (0,T] $ instead of $ [0,T] $.)
In addition, we know from Proposition~\ref{lem:ito} that for $ t \leq \Sigma^2 \frac{\tau^N}{N} $,
\begin{equation*}
	d X^{N,k}_t = \frac{1}{\Sigma^2} B^k\left(U^N_{\frac{Nt}{\Sigma^2}}\right) dt + d \tilde{Y}^{N,k}_t,
\end{equation*}
where $ (\tilde{Y}^{N,k}, t \geq 0, k \in [K]) $ is a continuous martingale with quadratic variation process
\begin{equation*}
	d [\tilde{Y}^{N,k}, \tilde{Y}^{N,k'}]_t = \frac{1}{\Sigma^2} \mathscr{C}_{k,k'}\left(U^{N}_{\frac{Nt}{\Sigma^2}}\right) dt.
\end{equation*}
Then, recall from the proof of Proposition~\ref{prop:stability_p}, that the
event $E_{Katz}^N$ defined in \eqref{eq:E_Katz} has a probability that tends to one
as $N$ tends to $\infty$.
Moreover, on this event, by Lemma~\ref{lemma:approx_B_C},
\begin{equation*}
	\sup_{c \Sigma^2 \frac{\log(N)}{N} \leq t \leq T} \left| B^k\left(U^N_{\frac{Nt}{\Sigma^2}}\right) \right| +\left| \frac{1}{\Sigma^2} \mathscr{C}_{k,k'}\left(U^{N}_{\frac{Nt}{\Sigma^2}}\right) - X^{N,k}_t \left( \ind_{k=k'} - X^{N,k'}_t \right) \right| = O\left( N^{-\beta/2} \right).
\end{equation*}
Theorem~\ref{thm:JS} then allows us to conclude that $ (X^N_t, t \in (0,T]) $ converges in distribution to a Wright-Fisher diffusion started from $ \theta_0 $.

\pagebreak
\appendix
\section*{Appendix}
\section{Convergence of continuous semimartingales} \label{appendix:JacodShiryaev}

We state here a result from \cite{jacod2013limit}, on the convergence of continuous semimartingales, which is used in the proof of our main result.

\begin{definition} \label{def:semimartingale}
	We say that $ (X_t, t \geq 0) $, taking values in $ \R^d $, is a continuous semimartingale with respect to some filtration $ (\mathcal{F}_t, t \geq 0) $ with characteristics $ (B, C) $ if
	\begin{enumerate}
		\item $ B = (B_t, t \geq 0) $ is a predictable process taking values in $ \R^d $ with locally finite variation and $ B_0 = 0 $,
		\item $ C = (C_t, t \geq 0) $ is a predictable and continuous process taking values in $ M_{d\times d}(\R) $ such that $ C_0 = 0 $ and, for all $ 0 \leq s \leq t $, $ C_t - C_s $ is non-negative and symmetric,
		\item the process $ (M_t, t \geq 0) $ defined by
		      \begin{equation*}
			      M_t := X_t - X_0 - B_t
		      \end{equation*}
		      is a continuous $ \mathcal{F}_t $-martingale with quadratic variation process $ C $.
	\end{enumerate}
\end{definition}

The following result is then stated as Theorem~IX.4.44 in \cite{jacod2013limit}.

\begin{thm} \label{thm:JS}
	Suppose that, for all $ N \geq 0 $, $ X^N = (X^N_t, t \geq 0) $ is a continuous semimartingale with characteristics $ (B^N, C^N) $, taking values in $ \R^d $ and that
	\begin{enumerate}
		\item $ X^N_0 \to X_0 $ as $ N \to \infty $ in distribution,
		\item there exists a continuous map $ B : \mathcal{C}(\R_+, \R^d) \to \mathcal{C}(\R_+, \R^d) $ such that, for all $ t \geq 0 $,
		      \begin{equation*}
			      \sup_{s \in [0,t]} | B^N_s - B(X^N)_s | \to 0,
		      \end{equation*}
		      in probability as $ N \to \infty $,
		\item there exists a continuous map $ C : \mathcal{C}(\R_+, \R^d) \to \mathcal{C}(\R_+, M_d(\R)) $ such that, for all $ t \geq 0 $,
		      \begin{equation*}
			      C^N_t - C(X^N)_t \to 0
		      \end{equation*}
		      in probability as $ N \to \infty $,
		\item there exists a continuous and increasing function $ F : \R_+ \to \R_+ $ such that for all $ i \in [d] $ and $ t \geq 0 $ and for all $ \omega \in \mathcal{C}(\R_+, \R^d) $,
		      \begin{align*}
			      [C(\omega)_t]_{i,i} \leq F_t, \quad \text{and}
			      \quad Var(B(\omega)_i)_t \leq F_t,
		      \end{align*}
		      where $ t \mapsto Var(\omega)_t $ is the variation process associated to $ \omega $,
		\item there exists a unique (in distribution) process $ (X_t, t \geq 0) $ such that $ X $ is a semimartingale with characteristics $ (B(X), C(X)) $.
	\end{enumerate}
	Then $ (X^N_t, t \geq 0) $ converges in distribution as $ N \to \infty $ to $ (X_t, t \geq 0) $, started from $ X_0 $.
\end{thm}

\section{Stochastic differential equations forced onto a manifold by a large drift} \label{sec:statement_katzenberger}

\label{sec:katzenberger}

Here we restate Katzenberger's result \citep[Theorem~6.3]{katzenberger1990solutions}.
Let $ \mathcal{U} \subset \mathbb{R}^d $ be an open set, let $ \mathbf{F} : \mathcal{U} \to \R^d $ be $ C^1 $ and assume that $ \Gamma := \mathbf{F}^{-1}(0) $ is a $ C^0 $ submanifold of $ \mathcal{U} $ of dimension $ m $.
Let $ \psi(x,t) $ be the solution of
\begin{equation} \label{def_flow_katz}
	\psi(x,t) = x + \int_{0}^{t} \mathbf{F}(\psi(x,s)) ds,
\end{equation}
define
\begin{equation*}
	\mathcal{U}_\Gamma := \lbrace x \in \mathcal{U} : \lim_{t \to \infty} \psi(x,t) \text{ exists and is in } \Gamma \rbrace,
\end{equation*}
and set, for $ x \in \mathcal{U}_\Gamma $,
\begin{equation*}
	\Phi(x) := \lim_{t \to \infty} \psi(x,t).
\end{equation*}

For all $ n \geq 1 $, let $ (\Omega^n, \mathcal{F}^n, (\mathcal{F}^n_t)_{t \geq 0}, \P) $ be a filtered probability space on which are defined the random variables $ Z_n $ and $ A_n $, where $ (Z_n(t), t \geq 0) $ is a c\`adl\`ag $ \R^e $-valued $ \mathcal{F}^n_t $-semimartingale with $ Z_n(0) = 0 $ and $ (A_n(t), t \geq 0) $ is a non-decreasing real-valued c\`adl\`ag $ \mathcal{F}^n_t $-adapted process with $ A_n(0) = 0 $.
Let $ \mathbf{G}_n : \mathcal{U} \to M_{d,e}(\R) $ be continuous and such that $ \mathbf{G}_n $ converges to a limit $ \mathbf{G} $, uniformly on compact subsets of $ \mathcal{U} $, as $ n \to \infty $.

Then, for each $ n \geq 1 $, let $ (X_n(t), t \geq 0) $ be an $ \R^d $-valued $ \mathcal{F}^n_t $-semimartingale satisfying, for any compact $ K \subset \mathcal{U} $,
\begin{equation} \label{katz:sde}
	X_n(t) = X_n(0) + \int_{0}^{t} \mathbf{G}_n(X_n(s)) dZ_n(s) + \int_{0}^{t} \mathbf{F}(X_n(s)) dA_n(s),
\end{equation}
for all $ t \leq \lambda_n(K) $ where
\begin{equation*}
	\lambda_n(K) := \inf \lbrace t \geq 0 : X_n(t^-) \notin \mathring{K} \text{ or } X_n(t) \notin \mathring{K} \rbrace.
\end{equation*}
For a process $ Y $, let $ Y^\lambda $ denote the process $ Y $ stopped at $ \lambda $, i.e.
\begin{equation*}
	Y^\lambda(t) := Y(t \wedge \lambda).
\end{equation*}
We further assume the following. For any compact $ K \subset \mathcal{U} $, $ Z_n^{\lambda_n(K)} $ admits a decomposition as the sum of a local martingale $ M_n $ and a finite variation process $ F_n $ such that there exists a sequence of stopping times $ (\tau^k_n, k \geq 1) $ with
\begin{equation*}
	\P(\tau^k_n \leq k) \leq 1/k,
\end{equation*}
and
\begin{equation*}
	\lbrace [M_n]_{t \wedge \tau^k_n} + T_{t \wedge \tau^k_n}(F_n), n \geq 1 \rbrace
\end{equation*}
is uniformly integrable for all $ t \geq 0 $ and $ k \geq 1 $, where $ T_t(Y) $ denotes the total variation of $ Y $ on the interval $ [0,t] $.
Moreover, suppose that, for any $ \varepsilon > 0 $ and $ T > 0 $,
\begin{equation*}
	\lim_{\gamma \to 0} \: \limsup_{n \to \infty} \: \P\left(\sup_{0 \leq t \leq T} (T_{t+\gamma}(F_n)-T_t(F_n)) > \varepsilon\right) = 0.
\end{equation*}
Also assume that, for all $ 0 < \varepsilon < T $ and any compact $ K \subset \mathcal{U} $,
\begin{equation*}
	\inf_{0 \leq t \leq T \wedge \lambda_n(K) - \varepsilon} A_n(t + \varepsilon) - A_n(t) \to \infty
\end{equation*}
in distribution as $ n \to \infty $.
Finally, assume that, for any compact $ K \subset \mathcal{U} $ and any $ T > 0 $,
\begin{align*}
	\sup_{0 \leq t \leq T \wedge \lambda_n(K)} | \Delta Z_n(t) | + | \Delta A_n(t) | \to 0,
\end{align*}
in distribution as $ n \to \infty $, where $ \Delta Y(t) := Y(t) - Y(t^-) $.

Katzenberger's main result (Theorem~6.3 in \cite{katzenberger1990solutions}) is the following.

\begin{thm} \label{thm:katzenberger}
	Assume, in addition to all the above, that $ \Gamma $ is $ C^2 $ and that, for all $ y \in \Gamma $, the Jacobian matrix of $ \mathbf{F} $ at $ y $ has $ d-m $ eigenvalues with negative real parts.
	Assume that $ \Phi : \mathcal{U}_\Gamma \to \Gamma $ is $ C^2 $ and that $ X_n(0) $ converges in distribution to $ X(0) $ which takes values in $ \mathcal{U}_\Gamma $.
	Then define
	\begin{equation} \label{def:Yn}
		Y_n(t) := X_n(t) - \psi(X_n(0), A_n(t)) + \Phi(X_n(0)),
	\end{equation}
	and, for any compact $ K \subset \mathcal{U} $,
	\begin{equation*}
		\mu_n(K) := \inf \lbrace t \geq 0 : Y_n(t^-) \notin \mathring{K} \text{ or } Y_n(t) \notin \mathring{K} \rbrace.
	\end{equation*}
	Then, for any compact $ K \subset \mathcal{U} $, $ (Y^{\mu_n(K)}, Z_n^{\mu_n(K)}, \mu_n(K)) $ is tight in $ D([0,\infty), \R^d \times \R^e) \times [0,\infty] $, and any limit of a converging subsequence $ (Y, Z, \mu) $ is such that $ (Y, Z) $ is a continuous semimartingale, $ Y $ takes values in $ \Gamma $,
	\begin{equation} \label{mu_liminf}
		\mu \geq \inf \lbrace t \geq 0 : Y(t) \notin \mathring{K} \rbrace
	\end{equation}
	almost surely and
	\begin{equation} \label{katz_limiting_sde}
		Y(t) = Y(0) + \int_{0}^{t \wedge \mu} \partial \Phi(Y(s)) \mathbf{G}(Y(s)) dZ(s) + \frac{1}{2} \sum_{ijk\ell} \int_{0}^{t \wedge \mu} \partial_{ij} \Phi(Y(s)) \mathbf{G}_{ik}(Y(s)) \mathbf{G}_{j\ell}(Y(s)) d[Z^k, Z^l]_s.
	\end{equation}
\end{thm}

\begin{rem}
	By the definition of $ \Phi $ and by the assumptions on $ A_n $,
	\begin{equation*}
		\Phi(X_n(0)) - \psi(X_n(0), A_n(t)) \to 0
	\end{equation*}
	in probability as $ n \to \infty $ for any $ t > 0 $.
	As a result, the convergence of $ Y_n $ on $ [0,T] $ yields the convergence (to the same limit) of $ X_n $ on any interval of the form $ [\varepsilon, T] $ for $ 0 < \varepsilon < T $.
\end{rem}

\section{Regularity of Katzenberger projections}
\label{sec:proof_regularity}
\begin{proof}[Proof of Lemma~\ref{lem:reg_katzenberger}]
	First, observe that the vector field $b$ belongs to $\mathcal{C}^3(\mathcal{B})$,
	so that, by \cite[Theorem 2.11]{teschl2012ordinary},
	\[
		\phi : (t, v_0) \mapsto \phi(t, v_0) \in \mathcal{C}^3(\mathbb{R}_+ \times \mathcal{B})
		\quad \text{and} \quad
		\varphi : (t, v_0, w_0) \mapsto \varphi(t, v_0, w_0) \in \mathcal{C}^3(\mathbb{R}_+ \times \mathcal{B} \times \mathbb{R}^E).
	\]
	Next, we show that for all $v_0 \in \mathcal{B}$ and $w_0 \in \mathbb{R}^E$, the derivative $\nabla_v \varphi(t, v_0, w_0)$ converges to a limit as $t \to \infty$. Recall that for each $x\in \Rv$, the partial derivative $\frac{\partial \varphi}{\partial v_x}(t, v_0, w_0)$ solves the ODE
	\begin{equation}
		\label{eq:dvphi_1}
		\frac{d}{dt} \left[\frac{\partial \varphi}{\partial v_x}(t, v_0, w_0)\right]
		= F(\phi(t, v_0)) \left[\frac{\partial \varphi}{\partial v_x}(t, v_0, w_0)\right]
		+ DF(\phi(t, v_0)) \left(\frac{\partial \phi}{\partial v_x}(t, v_0)\right) \varphi(t, v_0, w_0),
	\end{equation}
	with initial condition $\frac{\partial \varphi}{\partial v_x}(0, v_0, w_0) = 0$.
	Similarly, $\frac{\partial \phi}{\partial v_x}(t, v_0)$ solves
	\[
		\frac{d}{dt} \left[\frac{\partial \phi}{\partial v_x}(t, v_0)\right]
		= J_b(\phi(t, v_0)) \frac{\partial \phi}{\partial v_x}(t, v_0),
	\]
	with initial condition $\frac{\partial \phi}{\partial v_x}(0, v_0) = e_x$.
	Duhamel's formula then yields
	\begin{equation}
		\label{eq:d1}
		\frac{\partial \varphi}{\partial v_x}(t, v_0, w_0)
		= \int_0^t e^{(t-s)F(\tilde{h})}
		G_1(s,v_0,w_0)  \, ds,
	\end{equation}
	with
	\[
		G_1(s,v_0,w_0)=
		(F(\phi(s, v_0)) - F(\tilde{h}))\frac{\partial \varphi}{\partial v_x}(s, v_0, w_0)
		+ DF(\phi(s, v_0)) \left(\frac{\partial \phi}{\partial v_x}(s, v_0)\right) \varphi(s, v_0, w_0).
	\]

	Using the continuity of the Jacobian $J_b$ at $\tilde{h}$ together with \eqref{eq:HB1},
	we can choose $t_0>0$ such that for all $t \geq t_0$ and all $v_0\in \B$,
	the maximum real part of the eigenvalues of $J_b(\phi(t, v_0))$ is less than $\Re(\tilde{\lambda}_1)/2$. As a consequence, $\frac{\partial \phi}{\partial v_x}(t, v_0)$ decays exponentially to zero as $t \to \infty$,
	uniformly in $v_0\in \B$.
	Since $F$ is continuously differentiable and $s \mapsto \phi(s, v_0)$ is bounded, we obtain
	\begin{equation}
		\label{eq:exp_decay_DF}
		\left\| DF(\phi(s, v_0)) \left( \frac{\partial \phi}{\partial v_x}(s, v_0) \right) \right\| \leq M e^{-\gamma s}
	\end{equation}
	for some constants $M, \gamma > 0$. Putting this together with
	\eqref{eq:bound_F} and \eqref{eq:bound_vphi0}, we obtain
	\begin{equation}
		\label{eq:bound_normdphi}
		\left\|\frac{\partial \varphi}{\partial v_x}(t,v_0,w_0)\right\|
		\leq M\left(1+\int_0^t e^{-\gamma s}\left\|\frac{\partial \varphi}{\partial v_x}(s,v_0,w_0)\right\|\right)ds,
	\end{equation}
	which implies, by Grönwall's inequality, that the left-hand side of the above is uniformly
	bounded in $t\in[0,\infty)$ and $v_0\in\B$.
	{We then use the Perron-Frobenius decomposition \eqref{eq:PR_dec} to show that the right-hand side of \eqref{eq:d1} converges as $t\to \infty$. As in the proof of Proposition~\ref{prop:scalar_form},
	we first see from  \eqref{eq:bound_F}, \eqref{eq:bound_vphi0} and
	\eqref{eq:exp_decay_DF} that
	\[
		\int_{0}^{t}PG_1(s,v_0,w_0)ds
	\]
	is absolutely convergent (and thus convergent) when $t\to\infty$. Then,
	by an argument similar to that used to prove \eqref{eq:I2} (combined with \eqref{eq:bound_F}, \eqref{eq:bound_vphi0} and
	\eqref{eq:exp_decay_DF}),  one can show that
	\[
		\int_{0}^{t}R_{t-s}G_1(s,v_0,w_0)ds\xrightarrow[t\to\infty]{}0.
	\]}
	Using the same bounds, one can check that
	the remainder of the integral converges to $0$, uniformly in $v_0\in\B$. This implies that $v_0 \mapsto \theta(v_0, w_0)$ is differentiable on $\mathcal{B}$.

	The second and third order derivatives are handled in a similar way.
	We know that $\frac{\partial^2 \varphi}{\partial v_x \partial v_y}(t,v_0,w_0)$ is solution to the ODE
	\begin{align*}
		\frac{d}{dt}\left[\frac{\partial^2 \varphi}{\partial v_x \partial v_y}(t,v_0,w_0)\right]
		 & = F(\phi(t,v_0))\left[\frac{\partial^2 \varphi}{\partial v_x \partial v_y}(t,v_0,w_0)\right]
		+ G_2(t,v_0,w_0),
	\end{align*}
	with
	\begin{align*}
		G_2(t,v_0,w_0) & =
		DF(\phi(t,v_0))\left(\frac{\partial \phi}{\partial v_x}(t,v_0)\right)\frac{\partial \varphi}{\partial v_y}(t,v_0,w_0)                +DF(\phi(t,v_0))\left(\frac{\partial \phi}{\partial v_y}(t,v_0)\right)\frac{\partial \varphi}{\partial v_x}(t,v_0,w_0)        \\
		               & + D^2F(\phi(t,v_0))\left(\frac{\partial \phi}{\partial v_x}(t,v_0), \frac{\partial \phi}{\partial v_y}(t,v_0)\right)\varphi(t,v_0,w_0)  + DF(\phi(t,v_0))\left(\frac{\partial^2 \phi}{\partial v_x \partial v_y}(t,v_0)\right)\varphi(t,v_0,w_0),
	\end{align*}
	and with initial condition $\frac{\partial^2 \varphi}{\partial v_x \partial v_y}(0,v_0,w_0)=0$.
	Moreover,
	and $\frac{\partial^2 \phi}{\partial v_x \partial v_y}(t,v_0)$ is solution
	to
	\[
		\frac{d}{dt}\left[\frac{\partial \phi}{\partial v_x \partial v_y}(t,v_0)\right]
		=J_b(\phi(t,v_0))\frac{\partial^2 \phi}{\partial v_x \partial v_y}(t,v_0)
		+D^2b(\phi(t,v_0))\left(\frac{\partial \phi}{\partial v_x}(t,v_0), \frac{\partial \phi}{\partial v_y}(t,v_0)\right)
	\]
	with  $\frac{\partial \phi}{\partial v_x \partial v_y}(0,v_0)=0$.
	Applying Grönwall's lemma, we obtain that the norm of $\frac{\partial \phi}{\partial v_x \partial v_y}(t,v_0)$ decays exponentially. This result relies on
	\eqref{eq:HB1} and the continuity of $J_b$. Writing Duhamel's
	formula for $\frac{\partial^2\varphi}{\partial v_x \partial v_y}$ and
	substituting the bounds from  \eqref{eq:bound_F} and \eqref{eq:bound_vphi0}
	allows us  to establish an inequality analogous to \eqref{eq:bound_normdphi}  for
	$\frac{\partial^2\varphi}{\partial v_x \partial v_y}$. Hence,
	$\|\frac{\partial^2\varphi}{\partial v_x \partial v_y}\|$ is uniformly
	bounded in $t\in[0,\infty)$ and $v_0\in\B$. Reinjecting this estimate in Duhamel's
	formula shows that $\nabla^2  \varphi(t,v_0,w_0)$
	converges to a limit as $t\to\infty$, uniformly in $v_0\in \B$.
	The third derivative can be treated similarly using the same technique.
\end{proof}

\bibliography{effective_population_size,biblio_raph}
\bibliographystyle{plain}
\end{document}